\documentclass[a4paper, reqno]{amsart} 
\usepackage{amsmath}
\usepackage{amsthm}
\usepackage{enumitem}   
\usepackage[hidelinks,colorlinks=true]{hyperref}
\hypersetup{allcolors=[rgb]{0,0.31,0.62}}
\usepackage[utf8]{inputenc}
\usepackage{stix}       
\usepackage{thmtools}   
\usepackage{tikz}
\usetikzlibrary{matrix} 

\usetikzlibrary{decorations.markings}

\usepackage{mathrsfs}
\usepackage{caption}
\usepackage{algorithm}
\usepackage{algpseudocode}

\usepackage{spectralsequences}
\sseqset{homological Serre grading}
\sseqset{classes = {draw = none}}

\tikzstyle{vertex}=[circle,fill=black!25,minimum size=20pt,inner sep=0pt]

\tikzset{->-/.style={decoration={
  markings,
  mark=at position .5 with {\arrow{>}}},postaction={decorate}}}

\DeclareMathOperator{\id}{id}

\DeclareMathOperator{\Hom}{Hom}
\DeclareMathOperator{\op}{op}
\DeclareMathOperator{\Ima}{im}
\DeclareMathOperator{\Loc}{Loc}
\DeclareMathOperator{\Ent}{Ent}
\DeclareMathOperator{\Fac}{Fac}
\DeclareMathOperator{\Flo}{Flo}

\DeclareMathOperator{\cof}{lk}
\DeclareMathOperator{\Tot}{Tot}
\DeclareMathOperator{\NN}{\overline{N}}
\DeclareMathOperator{\DN}{N\overline{N}}
\DeclareMathOperator{\diag}{diag}
\newcommand{\R}{\mathbb{R}} 
\newcommand{\Z}{\mathbb{Z}}
\newcommand{\Q}{\mathbb{Q}}

\newcommand{\s}{\mathscr{S}}
\newcommand{\p}{\mathscr{P}}
\newcommand{\N}{\mathcal{N}}
\newcommand{\bigo}{\mathcal{O}}
\newcommand{\Set}{\mathsf{Set}}
\newcommand{\sSet}{\mathsf{sSet}}

\newcommand{\Cat}{\mathsf{Cat}}

\newcommand{\diagram}[3]{\matrix (#1) [matrix of math nodes,row
  sep={#2},column sep={#3},text height=1.5ex,text
  depth=0.25ex]}

\theoremstyle{plain}
\newtheorem{theorem}{Theorem}[section]
\newtheorem{proposition}[theorem]{Proposition}
\newtheorem{corollary}[theorem]{Corollary}
\newtheorem{lemma}[theorem]{Lemma}

\theoremstyle{definition}
\newtheorem{definition}[theorem]{Definition}
\newtheorem{example}[theorem]{Example}

\newtheorem{thmx}{Theorem}

\makeatletter
\def\namedlabel#1#2{\begingroup
  #2%
  \def\@currentlabel{#2}%
  \phantomsection\label{#1}\endgroup
}
\makeatother

\title{The discrete flow category: structure and computation}
\author{Bjørnar Gullikstad Hem}
\thanks{Email: bjornar.hem@epfl.ch }
\date{\today}

\begin{document}
\maketitle

\begin{abstract}
In this article, we use concepts and methods from the theory of simplicial sets to study discrete Morse theory. We focus on the discrete flow category introduced by Vidit Nanda, and investigate its properties in the case where it is defined from a discrete Morse function on a regular CW complex. We design an algorithm to efficiently compute the Hom posets of the discrete flow category in this case. Furthermore, we show that in the special case where the discrete Morse function is defined on a simplicial complex, then each Hom poset has the structure of a face poset of a regular CW complex. Finally, we prove that the spectral sequence associated to the double nerve of the discrete flow category collapses on page 2.

\end{abstract}

\setcounter{tocdepth}{1}
\tableofcontents

\section{Introduction}

Differentiable functions on smooth manifolds allow us to deduce properties of the manifolds' topology, through the use of Morse theory. For Morse theory to apply, the differential function has to satisfy the \emph{Morse condition}. Such functions are called \emph{Morse functions}, and they are in many ways abundant. A Morse function can be viewed as a ``height function'' on a manifold that allows you to decompose the manifold into smaller, more manageable parts. Morse theory has had vast applications, not only in topology, but also in other areas, such as in the study of dynamical systems \cite{smale}.

The combinatorial counterpart of Morse theory, discrete Morse theory, was introduced by Robin Forman \cite{Forman1998}, and studies \emph{discrete Morse functions} on \emph{simplicial complexes}. Like Morse theory, discrete Morse theory allows us to deduce information about the topology of a simplicial complex. As the name suggests, the data of a discrete Morse function is discrete, making discrete Morse theory suitable for computer applications. Discrete Morse theory has seen much use in computer applications in later years, particularly in topological data analysis (TDA).
\textcolor{black}{As an example, in \cite{Bauer_2021}, Bauer directly utilizes discrete Morse theory to create an efficient algorithm to compute the persistent homology of Vietoris-Rips complexes.}


One of the main foundations for this article is a paper by Nanda \cite{Nanda} that introduces the \emph{discrete flow category}, an analogue in discrete Morse theory for the \emph{flow category} of Cohen, Jones and Segal \cite{CohenJonesSegal}. The flow category is a category consisting of critical points and flow lines of a Morse function, and whose classifying space captures the homotopy type of the manifold on which the Morse function is defined. Likewise, the discrete flow category consists of critical \emph{cells} and \emph{gradient paths} of a discrete Morse function, and has a classifying space with the homotopy type of the simplicial complex on which the discrete Morse function is defined. While the flow category is a \emph{topological category}, meaning that the Hom sets are topological spaces, the discrete flow category is a \emph{p-category}, meaning that the Hom sets are posets. Constructing the classifying space of either of these categories utilizes concepts from simplicial sets, inspiring the further study of discrete Morse theory from the perspective of simplicial sets.

Another inspiration for this article is a paper by Vaupel, Hermansen and Trygsland on section complexes of simplicial height functions \cite{sectionComplexes}. In the paper, a bisimplicial set is constructed from a \emph{height function} on a simplicial set. This bisimplicial set gives a spectral sequence that can be used to compute the homology of the simplicial set, and it is shown that under certain conditions the spectral sequence collapses on the second page.

In this article we explore Nanda's discrete flow category (in the special case where it is defined from a discrete Morse function), and prove several properties of it. In particular, we show that in the case where the discrete Morse function is defined on a simplicial complex (not a general CW complex), then the Hom posets in the discrete flow category are CW posets, meaning that they have the structure of a face poset of a regular CW complex. We will refer to this as Theorem A.
\begin{thmx}\label{theoremA}
    Let $C$ be the discrete flow category of a discrete Morse function on a simplicial complex. Then the for all objects $w$ and $z$ in $C$, the poset $\Hom_{C}(w, z)^{\op}$ is the face poset of some regular CW complex.
\end{thmx}
\noindent This result gives a simpler way of realizing the Hom posets as topological spaces: instead of taking the geometric realization of the nerve, one can take the corresponding regular CW complex.

 \textcolor{black}{
Note that \autoref{theoremA} bears similarities with work done in \cite{nanda2018discrete}, in particular with \cite[Proposition 2.45]{nanda2018discrete} and \cite[Corollary 3.41]{nanda2018discrete}. There are, however, several important differences between \autoref{theoremA} and these results. Firstly, the flow category defined in \cite{nanda2018discrete} is different from the one we use here and in \cite{Nanda} (in particular, in \cite{nanda2018discrete}, flow paths cannot have intermediate cells). Secondly, \cite[Corollary 3.41]{nanda2018discrete} concerns the poset of all flow paths in a CW complex, while \autoref{theoremA} concerns the poset of flow paths between two given cells. Finally, the results in \cite{nanda2018discrete} holds for all regular CW complexes, while \autoref{theoremA} holds when the regular CW complex is a simplicial complex, and a counterexample for the general case is given in \autoref{ex:theoremAcounter}.
}

Furthermore, we construct an algorithm, \autoref{algorithm_1}, to compute the Hom posets of the discrete flow category, where the input is a discrete Morse function defined on a regular CW complex.
For the algorithm, we use several of the results needed for proving \autoref{theoremA}.
We provide a Python implementation for the algorithm, which can be found at \mbox{\url{https://github.com/bjornarhem/discrete_flow}} .

Finally, we investigate the spectral sequence associated to the \emph{double nerve} of the discrete flow category, a bisimplicial set whose realization is the classifying space of the category.
As in the paper by Vaupel et al. \cite{sectionComplexes}, this spectral sequence computes the homology of the regular CW complex from which we define the discrete flow category. 
Classifying spaces of p-categories are often complicated and hard to compute, and this is also the case for the discrete flow category.
However, as we prove in this article, the spectral sequence associated to its double nerve has a particularly nice structure that causes it to collapse on page 2. We will refer to this as Theorem B.
\begin{thmx}\label{theoremB}
    The spectral sequence associated to the double nerve of a discrete flow category collapses on page 2.
\end{thmx}
The definition of the mentioned spectral sequence will be made formal in the text (in particular, there are two possible choices). We also show how you can ignore degeneracies when computing the spectral sequence, which makes the computation relatively simple. Finally, we provide several examples of computing the spectral sequence, in which we make use of \autoref{algorithm_1} to compute the Hom sets.
We also provide a Python implementation of the spectral sequence computation at \mbox{\url{https://github.com/bjornarhem/discrete_flow} }, that outputs the spectral sequence pages given any discrete Morse function.

To prove \autoref{theoremB}, we generalize the concept of \emph{simplicial collapse} to simplicial sets. We then introduce \emph{unique factorization categories}, a class of categories in which all morphisms have a unique decomposition into ``indecomposable'' morphisms. Finally, we apply our generalization of simplicial collapse to prove the following result.
\begin{thmx}\label{theoremC}
    Let $C$ be a unique factorization category. Then the nerve of $C$ deformation retracts to a simplicial set whose $n$--simplices are all degenerate for $n > 1$.
\end{thmx}
\noindent This theorem is then used in the proof of \autoref{theoremB}.

\subsection*{Outline}

Section 2 and 3 are dedicated to reviewing the theory of simplicial sets and of discrete Morse theory, respectively, which will be needed for the following sections. Section 4 contains the necessary preliminaries on p-categories, and Section 5 is a review of the discrete flow category of a discrete Morse function. In Section 6, we prove several results on the Hom posets of the discrete flow category, and conclude with constructing \autoref{algorithm_1} and proving \autoref{theoremA}. In Section 7, we introduce simplicial collapse on simplicial sets and unique factorization categories, and prove \autoref{theoremC}. Finally, in Section 8, we explain how the discrete flow category gives rise to a spectral sequence, and prove \autoref{theoremB}.

\subsection*{Acknowledgments}

This work is based on my master thesis, that I wrote for NTNU in Spring 2023.
I would like to thank my master thesis supervisors, Marius Thaule and Melvin Vaupel, who have
both been incredibly helpful and supportive. They have gone above and beyond,
both in encouraging me to explore my ideas and in helping me seek answers to
the countless questions I have asked them. Their feedback, which has consistently
been both thorough and constructive, has been invaluable in this work.
I would also like to thank my PhD supervisor, Kathryn Hess Bellwald, for providing feedback in the final steps of writing this article.
 \textcolor{black}{Finally, I would like to thank the two referees of this paper, who both made insightful and useful comments.}

\section{Preliminaries on simplicial sets}

In this section we define our notation related to simplicial sets, and state some results that we will need.
 \textcolor{black}{For a thorough exposition of simplicial sets, see e.g., \cite{GoerssJardine} or \cite{Friedman}.}

A \emph{simplicial set} is a functor
\begin{equation*}
    \Delta^{\op} \to \Set,
\end{equation*}
where $\Delta$ is the simplex category and $\Set$ is the category of sets. A morphism of simplicial sets is a natural transformation between two such functors, and we denote the category of simplicial sets by $\sSet$. The $i$th face map is denoted by $d_i$, and the $i$th degeneracy map is denoted by $s_i$. For a simplicial set $X$, we denote its geometric realization by $|X|$. For a small category $C$, we denote its nerve by $\N(C)$.

%
%
%
%
%
%
%

\subsection{Barycentric subdivision}

We here define what we mean by the barycentric subdivision of a regular CW complex. Recall that a regular CW complex is a CW complex where each attaching map is a homeomorphism onto its image.
\begin{definition}
    Let $X$ be a regular CW complex. The barycentric subdivision of $X$, denoted $T(X)$, is a simplicial complex whose vertices are the cells of $X$ and whose $n$--simplices are sequences of distinct cells $(\sigma_0, \dots, \sigma_n)$ such that $\sigma_0 \subseteq \dots \subseteq \sigma_n$.
\end{definition}
An example of a barycentric subdivision is given in \autoref{fig:bary_subdiv}.

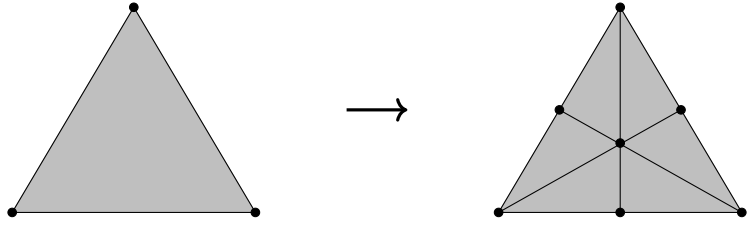
\begin{figure}[htbp]
    \centering
    \begin{tikzpicture}[scale=.8]
        \filldraw[lightgray] (0,0) -- (4,0) -- (2,3.4) -- cycle;
        \filldraw (0,0) circle (2 pt);
        \filldraw (4,0) circle (2 pt);
        \filldraw (2,3.4) circle (2 pt);
        \draw (0,0) -- (4,0);
        \draw (0,0) -- (2,3.4);
        \draw (4,0) -- (2,3.4);

        \draw[very thick, ->] (5.5, 1.7) -- (6.5, 1.7);

        \filldraw[lightgray] (8,0) -- (12,0) -- (10,3.4) -- cycle;
        \filldraw (8,0) circle (2 pt);
        \filldraw (10,0) circle (2 pt);
        \filldraw (12,0) circle (2 pt);
        \filldraw (11,1.7) circle (2 pt);
        \filldraw (10,3.4) circle (2 pt);
        \filldraw (9,1.7) circle (2 pt);
        \filldraw (10,1.15) circle (2 pt);
        \draw (8,0) -- (12,0);
        \draw (8,0) -- (11,1.7);
        \draw (8,0) -- (10,3.4);
        \draw (12,0) -- (10,3.4);
        \draw (12,0) -- (9,1.7);
        \draw (10,0) -- (10,3.4);
    \end{tikzpicture}
    \caption{The barycentric subdivision of the simplicial complex $\Delta^2$.}
    \label{fig:bary_subdiv}
\end{figure}

\begin{theorem}
    If $X$ is a regular CW complex, then the geometric realization of the barycentric subdivision, $|T(x)|$, is homeomorphic to $X$.
\end{theorem}

For a proof, see \cite[pp.\ 80--81]{Lundell}.

\subsection{Simplicial homology}
\label{sec:simplicial_homology}

%
%
%


In this section we state some results on simplicial homology that will be needed in later sections. We get the following result by combining Theorem 2.1 and Theorem 2.4 in \cite[Chapter III]{GoerssJardine}.

\begin{theorem}\label{thm:simpl_homology_degen}
    Let $C$ be the (unnormalized) simplicial chain complex associated to a simplicial set, and $D$ its subcomplex of degeneracies. Then, for all $n$ there is an isomorphism
    \begin{equation*}
        H_n(C) \cong H_n(C / D),
    \end{equation*}
    induced by the projection map $p \colon C \to C/D$.
\end{theorem}

The following corollary follows immediately from applying the long exact sequence of homology to the short exact sequence of chain complexes $0 \to D \to C \to C/D \to 0$.

\begin{corollary}\label{cor:simpl_homology_degen}
    Let $C$ be the (unnormalized) simplicial chain complex associated to a simplicial set, and $D$ its subcomplex of degeneracies. Then,
    \begin{equation*}
        H_n(D) \cong 0,
    \end{equation*}
    for all $n$.
\end{corollary}

\subsection{Bisimplicial sets}

A \emph{bisimplicial set} is a functor
\begin{equation*}
    \Delta^{\op} \times \Delta^{\op} \to \Set,
\end{equation*}
or equivalently,
\begin{equation*}
    \Delta^{\op} \to \sSet.
\end{equation*}

For a bisimplicial set $X$, the \emph{diagonal} of $X$, denoted $\diag X$, is the simplicial set given by:
\begin{itemize}
    \item $(\diag X)([n]) = X([n],[n])$,
    \item For $\theta \colon [m] \to [n]$, $(\diag X)(\theta) = X(\theta,\theta)$.
\end{itemize}
We define the \emph{geometric realization} of a bisimplicial set $X$, denoted $|X|$, as the realization of the diagonal, i.e., $|X| \coloneq |\diag X|$.

\section{Discrete Morse theory}

In this section, we briefly summarize the core ideas and definitions from discrete Morse theory.

In Morse theory, one studies certain differentiable functions, called \emph{Morse functions}, on smooth manifolds. These Morse functions allow us to deduce information about the topology of the manifold. For more information about Morse theory, see e.g. \cite{Milnor}. Discrete Morse theory is a discrete analogue of this, where one studies real-valued functions on CW complexes. The functions assign a single real value to each cell of the CW complex, and hence we get a discrete set of values.

 \textcolor{black}{We will only define discrete Morse functions for \emph{regular} CW complexes. For the general definition, that works for all CW complexes, see \cite{Forman1998}.}
Note also that one often consider only discrete Morse functions on simplicial complexes, which can be considered a special class of regular CW complexes.

\begin{definition}
    Let $X$ be a regular CW complex. A \emph{discrete Morse function} on $X$ is a function $f \colon X \to \R$ that assigns a real value to each cell of $X$ such that for each $k$--cell $x$, the following conditions hold.
    \begin{enumerate}
        \item There is at most one $(k+1)$--cell $y$ such that $x \subseteq y$ and $f(x) \ge f(y)$.
        \item There is at most one $(k-1)$--cell $y$ such that $y \subseteq x$ and $f(y) \ge f(x)$.
    \end{enumerate}
\end{definition}

\begin{definition}\label{def:critical_cell}
     \textcolor{black}{A $k$--cell} $x$ is called \emph{critical} with respect to a discrete Morse function $f\colon X \to \R$ if the following conditions hold.
    \begin{enumerate}
        \item There are no $(k+1)$--cells $y$ such that $x \subseteq y$ and $f(x) \ge f(y)$, and
        \item There are no $(k-1)$--cells $y$ such that $y \subseteq x$ and $f(y) \ge f(x)$.
    \end{enumerate}
\end{definition}
 \textcolor{black}{
Cells that are not critical are called \emph{regular}. Note that for a discrete Morse function, it is impossible for both of these conditions to fail to hold for a single cell (this is often called the exclusion lemma) \cite[Lemma 2.5]{Forman1998}. An example of a discrete Morse function is given in \autoref{fig:dmf_example}.
}

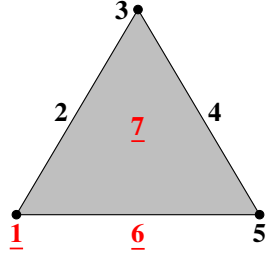
\begin{figure}[htbp]
    \centering
    \begin{tikzpicture}[scale=.8]
        \filldraw[lightgray] (0,0) -- (4,0) -- (2, 3.4) -- cycle;
        \draw (2, 1.4) node[red]{$\mathbf{\underline{7}}$};
        
        \draw (0, 0) -- (4, 0) node[midway, anchor=north, red]{$\mathbf{\underline{6}}$};
        \draw (0, 0) -- (2, 3.4) node[midway, anchor=east]{$\mathbf{2}$};
        \draw (4, 0) -- (2, 3.4) node[midway, anchor=west]{$\mathbf{4}$};
        
        \filldraw (0, 0) circle (2 pt) node[anchor=north, red]{$\mathbf{\underline{1}}$};
        \filldraw (2, 3.4) circle (2 pt) node[anchor=east]{$\mathbf{3}$};
        \filldraw (4, 0) circle (2 pt) node[anchor=north]{$\mathbf{5}$};
    \end{tikzpicture}
    \caption{An example of a discrete Morse function.  \textcolor{black}{The critical values are underlined and colored red.}}
    \label{fig:dmf_example}
\end{figure}

\subsection{Gradient vector fields}

 \textcolor{black}{
For a CW complex $X$, we write $x^{(k)} \in X$ when $x$ is a $k$--dimensional cell in $X$.
\begin{definition}
    Let $f \colon X \to \R$ be a discrete Morse function, let $x \in X^{(k)}$ and $y \in X^{(k+1)}$. Then $\{x, y\}$ is called a \emph{regular pair} if $x \subseteq y$ and $f(x) \ge f(y)$.
\end{definition}
In discrete Morse theory, the essential information in a discrete Morse function is what its regular pairs are. Therefore, instead of talking about a discrete Morse function $f$, we often talk about its induced \emph{gradient vector field} $V_f$.}
\begin{definition}
    Let $f \colon X \to \R$ be a discrete Morse function. The \emph{induced gradient vector field} of $f$, denoted $V_f$, is defined as the set of \emph{regular pairs}, i.e.,
	 \textcolor{black}{\begin{equation*}
        V_f = \{\{x^{(k)}, y^{(k+1)}\} : k \in \Z_{\ge 0}, x \subseteq y, f(x) \ge f(y)\}.
    \end{equation*}}
\end{definition}
An example of a gradient vector field is given in \autoref{fig:gradient_vec_field}. Two discrete Morse functions that induce the same gradient vector field are said to be \emph{Forman equivalent}.

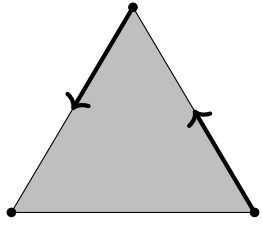
\begin{figure}
    \centering
    \begin{tikzpicture}[scale=.8]
        \filldraw[lightgray] (0,0) -- (4,0) -- (2, 3.4) -- cycle;
        \draw (2, 1.4);
        
        \draw (0, 0) -- (4, 0);
        \draw (0, 0) -- (2, 3.4);
        \draw (4, 0) -- (2, 3.4);
        
        \filldraw (0, 0) circle (2 pt);
        \filldraw (2, 3.4) circle (2 pt);
        \filldraw (4, 0) circle (2 pt);
        
        \draw[ultra thick, ->] (2, 3.4) -- (1, 1.7);
        \draw[ultra thick, ->] (4, 0) -- (3, 1.7);
    \end{tikzpicture}
    \caption{The gradient vector field for the discrete Morse function in \autoref{fig:dmf_example}.}
    \label{fig:gradient_vec_field}
\end{figure}

Given a regular CW complex $X$, a \emph{discrete vector field} on $X$ is a set $V$ of mutually disjoint pairs $\{x, y\}$ such that $\dim y = \dim x + 1$. Given a discrete vector field $V$, a \emph{$V$--path} is a sequence of simplices $(x_0, y_0, x_1, \dots, y_m, x_{m+1})$  \textcolor{black}{such that $x_i^{(k)} \in X, y_i^{(k+1)} \in X$,} $x_{i+1} \subseteq y_i$, $\{x_i, y_i\} \in V$ and $x_{i+1} \neq x_i$.
Discrete vector fields and gradient vector fields are related in the following way.
\begin{theorem}{\cite[Theorem 9.3]{Forman1998}}
    A discrete vector field $V$ is a gradient vector field of some discrete Morse function if and only if there are no nontrivial $V$--paths that start and end on the same cell.
\end{theorem}

\subsection{Simplicial collapse}\label{sec:simplicial_collapse}

In this section, we restrict our attention to simplicial complexes, and define what is known as a \emph{simplicial collapse}.
A \emph{free face} in a simplicial complex is a $k$--simplex that is the face of exactly one $(k+1)$--simplex.  \textcolor{black}{If $x^{(k)} \in X$ is a free face,} and $y$ is its $(k+1)$--dimensional coface, then $\{x, y\}$ is called a \emph{free pair}.
Removing a free pair from a simplicial complex is called an \emph{elementary collapse}. 
A sequence of elementary collapses is called a \emph{compound collapse} or \emph{simplicial collapse}. If the simplicial complex $Y$ can be produced from a simplicial complex $X$ through a compound collapse, we say that $X$ \emph{collapses} to $Y$ and write $X \searrow Y$.
Of particular importance is the fact that elementary collapses, and as a consequence compound collapses, yield a deformation retract. That is, if $X \searrow Y$, then $X$ deformation retracts to $Y$.

An example of an elementary collapse is given in \autoref{fig:elementary_collapse}.
\begin{figure}[htbp]
    \centering
    \begin{tikzpicture}[scale=.7]
        \filldraw[lightgray] (0,0) -- (4,0) -- (2, 3.4) -- cycle;
        
        \draw (0, 0) -- (4, 0);
        \draw (0, 0) -- (2, 3.4);
        \draw (4, 0) -- (2, 3.4);
        
        \filldraw (0, 0) circle (2 pt);
        \filldraw (2, 3.4) circle (2 pt);
        \filldraw (4, 0) circle (2 pt);
        
        \draw[->, very thick] (1.5, 2.6) -- (2, 2.3);
        \draw[->, very thick] (1, 1.7) -- (2, 1.1);
        \draw[->, very thick] (.5, .9) -- (1, 0.6);
        
        \draw[very thick, ->, double] (5.5, 1.7) -- (6.5, 1.7);
        
        \draw (8, 0) -- (12, 0);
        \draw (12, 0) -- (10, 3.4);
        
        \filldraw (8, 0) circle (2 pt);
        \filldraw (10, 3.4) circle (2 pt);
        \filldraw (12, 0) circle (2 pt);
    \end{tikzpicture}
    \caption{An elementary collapse.}
    \label{fig:elementary_collapse}
\end{figure}
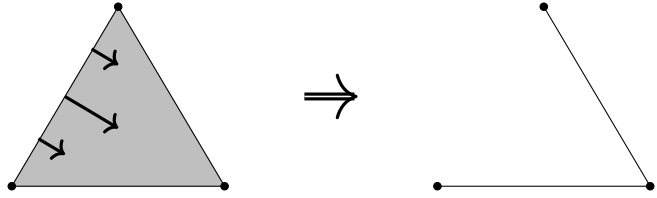

%
%

\section{p-categories}

 \textcolor{black}{
A \emph{p-category} is a category enriched in the category of posets.
In other words, a p-category is a category where the Hom sets have a poset structure satisfying the following property. For $f, f' \in \Hom(a, b)$ and $g, g' \in  \Hom(b, c)$ such that $f \Rightarrow f'$ and $g \Rightarrow g'$, we have that $g \circ f \Rightarrow g' \circ f'$.
Here and throughout the paper we use the symbol $\Rightarrow$ to denote a partial order relation between two morphisms.
}

A \emph{strict p-functor} $F \colon C \to D$ maps objects in $C$ to objects in $D$, and maps each Hom poset $\Hom_C(x,y)$ monotonically to $\Hom_D(F(x),F(y))$, such that
\begin{enumerate}
    \item $F(\id_x) = \id_{F(x)}$ for all objects $x$ in $C$, and
    \item $F(f \circ g) = F(f) \circ F(g)$ for all composable $f$ and $g$.
\end{enumerate}

A p-category can be viewed as a special case of a strict 2--category, where the 2--morphisms are partial order relations between morphisms. One can easily verify that this satisfies all the axioms of a strict 2--category. In particular, for a p-category viewed as a 2--category, there can only be zero or one 2--morphisms between two 1--morphism (either there is a partial order relation, or there is not).

 \textcolor{black}{
There are several ways to define the classifying space of a p-category (and more generally, a 2--category), but the different classifying spaces of a given p-category are all homotopy equivalent (\cite{Bullejos}, \cite{Carrasco}).}
In the rest of this section, we define the \emph{double nerve} of a p-category and let its geometric realization define the classifying space of a p-category.

\subsection{Simplicial categories}

A simplicial category is a simplicial object in $\Cat$, i.e., a functor\begin{equation*}
    \Delta^{\op} \to \Cat.
\end{equation*}

Composing such a functor $F \colon \Delta^{\op} \to \Cat$ with the nerve functor $\N \colon \Cat \to \sSet$ gives rise to a bisimplicial set $(\N \circ F)$. You can further take the diagonal of this bisimplicial set to get the simplicial set $\diag (\N \circ F)$.

\begin{example}\label{ex:simpl_cat_nerve}
Consider the following simplicial category $\s$.
\begin{itemize}
    \item $\s[0]$ is a category with two objects, $a$ and $b$, and two non-identity morphisms, $f \colon a \to b$ and $g \colon a \to b$.
    \item $\s[1]$ consists of degeneracies of $\s[0]$ and a single non-degenerate morphism $F \colon s_0 a \to s_0 b$, with $d_1 F = f$ and $d_0 F = g$.
    \item For $n \ge 2$, all objects and morphisms in $\s[n]$ are degenerate.
\end{itemize}

\autoref{fig:nerve_simpl_cat} illustrates the simplicial set $\diag(\N \circ \s)$ (omitting unimportant degenerate simplices). Observe that the geometric realization is a suspension of $|\Delta^1|$, i.e., homeomorphic to a disk $D^2$.

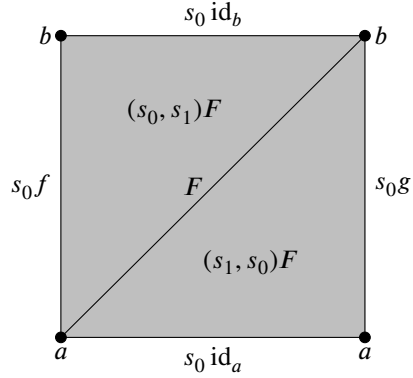
\begin{figure}[htbp]
    \centering
    \begin{tikzpicture}
        \filldraw[lightgray] (0,0) -- (4,0) -- (4, 4) -- (0, 4) -- cycle;
        
        \filldraw (0,0) circle (2 pt) node[anchor=north] {$a$};
        \filldraw (4,0) circle (2 pt) node[anchor=north] {$a$};
        \filldraw (0,4) circle (2 pt) node[anchor=east] {$b$};
        \filldraw (4,4) circle (2 pt) node[anchor=west] {$b$};

        \draw (0,0) -- (4,0) node[midway, anchor=north]{$s_0 \id_a$};
        \draw (0,4) -- (4,4) node[midway, anchor=south]{$s_0 \id_b$};
        \draw (0,0) -- (0,4) node[midway, anchor=east]{$s_0 f$};
        \draw (4,0) -- (4,4) node[midway, anchor=west]{$s_0 g$};
        \draw (0,0) -- (4,4) node[midway, anchor=east]{$F$};

        \draw (2.5, 1) node {$(s_1, s_0) F$};
        \draw (1.5, 3) node {$(s_0, s_1) F$};
    \end{tikzpicture}
    \caption{The simplicial set $\diag(\N \circ \s)$. Note that the top and bottom edges are degenerate, as $s_0 \id_a = (s_0, s_0) a$, and similar for $s_0 \id_b$. The left and right edges are, however, not degenerate as 1--simplices in $\diag(\N \circ \s)$.}
    \label{fig:nerve_simpl_cat}
\end{figure}

\end{example}

\subsection{The double nerve of a p-category}
\label{sec:double_nerve}
Given a p-category $C$, one can construct a simplicial category $\NN C$ as follows:

\begin{itemize}
    \item The objects in $(\NN C)[0]$ are the objects in $C$.
    
    \item For $n \ge 1$, the objects in $(\NN C)[n]$ are degeneracies of the objects in $(\NN C)[0]$.
    
    \item The morphisms $a \to b$ in $(\NN C)[0]$ are the morphisms $a \to b$ in $C$.
    
    \item For $n \ge 1$, the morphisms $(s_0)^n a \to (s_0)^n b$ in $(\NN C)[n]$ are $n$--simplices in $\N \left( \Hom_C (a, b) \right)$ (where we consider the poset $\Hom_C (a, b)$ as an ordinary category and take its nerve).
    
    In other words, morphisms in $\Hom_{(\NN C)[n]}\left( (s_0)^n a, (s_0)^n b\right)$ are sets of morphisms $\{f_0, \dots, f_n\}$ in $\Hom_C(a,b)$ such that 
    \begin{equation*}
        f_0 \Rightarrow f_1 \Rightarrow \dots \Rightarrow f_n
    \end{equation*}
    (here $(s_0)^n$ means $s_0$ applied $n$ times).

    The face and degeneracy maps are as in $\N\left( \Hom_C(a, b) \right)$.
    
    \item The composition of $\{f_0, \dots, f_n\} \colon (s_0)^n a \to (s_0)^n b$ with $\{g_0, \dots, g_n\} \colon (s_0)^n b \to (s_0)^n c$ is $\{g_0 \circ f_0, \dots, g_n \circ f_n\}$.
\end{itemize}

It's easy to verify that $g_0 \circ f_0 \Rightarrow g_1 \circ f_1 \Rightarrow \dots \Rightarrow g_n \circ f_n$, so that the composition rule is well-defined. It's also easily verified that this composition rule is associative, and that $\{\id_x, \dots, \id_x\}$ acts as the identity on $(s_0)^n x$. To show that $\NN C$ is a well-defined simplicial category, it remains to show that the face maps and degeneracy maps are functors between small categories. It's clear that the face map $d_i \colon (\NN C)[n] \to (\NN C)[n-1]$ sends a morphisms $\{f_0, \dots, f_n\} \in \Hom_{(\NN C)([n])}\left((s_0)^n a, (s_0)^n b\right)$ to $\Hom_{(\NN C)([n-1])}\left(d_i (s_0)^n a, d_i (s_0)^n b\right)$, as $d_i (s_0)^n = (s_0)^{n-1}$. Furthermore, $d_i$ commutes with composition:
\begin{align*}
    d_i & \left( \{g_0, \dots, g_n\} \circ \{f_0, \dots, f_n\} \right) \\
    & = d_i \{g_0 \circ f_0, \dots, g_n \circ f_n \} \\
    & = \{g_0 \circ f_0, \dots, g_{i-1} \circ f_{i-1}, g_{i+1} \circ f_{i+1}, \dots, g_n \circ f_n \} \\
    & = \{g_0, \dots, g_{i-1}, g_{i+1}, \dots, g_n \} \circ \{f_0, \dots, f_{i-1}, f_{i+1}, \dots, f_n \} \\
    & = \left( d_i \{g_0, \dots, g_n\}\right) \circ \left( d_i \{f_0, \dots, f_n\}\right)
\end{align*}
Hence, the face maps are functors. Similar computations for $s_i$ gives that the degeneracy maps are functors. In conclusion, $\NN C$ is a well-defined simplicial category.

\begin{example}
Let's consider the p-category $\p$ with two objects, $a$ and $b$, two non-identity morphisms, $f \colon a \to b$ and $g \colon a \to b$, and a single non-identity partial order relation $f \Rightarrow g$ (see also \autoref{fig:simple_p_cat} for an illustration).

\begin{figure}[htbp]
    \centering
    \begin{tikzpicture}[scale=.6]
        \draw[thick, ->] (-1.5, 0.5) .. controls (-0.5, 1.0) and (0.5, 1.0) .. (1.5, 0.5) node[midway, anchor=south]{$f$};
        \draw[thick, ->] (-1.5, -0.5) .. controls (-0.5, -1.0) and (0.5, -1.0) .. (1.5, -0.5) node[midway, anchor=north]{$g$};
        
        \filldraw (-2, 0) circle (2 pt) node[anchor=north]{$a$};
        \filldraw (2, 0) circle (2 pt) node[anchor=north]{$b$};

        \draw (0, 0) node{$\Downarrow$};
    \end{tikzpicture}
    \caption{The p-category $\p$.}
    \label{fig:simple_p_cat}
\end{figure}
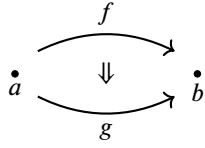

Let $\NN \p$ be the corresponding simplicial category as described above. Then the objects in $(\NN \p)[0]$ are $a$ and $b$, while the morphisms in $(\NN \p)[0]$ are $f$ and $g$, together with identities. There is a single non-degenerate morphism in $(\NN \p)[1]$, namely $F := \{f, g\}$, which has the faces $d_1 F = f$ and $d_0 F = g$. All other objects and morphisms are degenerate. Thus, $\NN \p$ is precisely the simplicial category $\s$ from \autoref{ex:simpl_cat_nerve}.
\end{example}

As mentioned in the previous section, we can  \textcolor{black}{compose} the simplicial category $\NN C$ with the nerve operation $\N$ to get a bisimplicial set. We will call this composite operation the \emph{double nerve}.

\begin{definition}
    Let $C$ be a p-category. The \emph{double nerve} of $C$, denoted $\DN C$, is the bisimplicial set $\N \circ (\NN C)$. The \emph{classifying space} of $C$ is the geometric realization of the double nerve, i.e., $|\DN C|$.
\end{definition}

 \textcolor{black}{
Note that in \cite{Bullejos} and \cite{Carrasco}, the intermediate simplicial category $\NN C$ ($\underline{\text{N}} C$ in their papers) is defined differently from here; there, the $n$--simplices are horizontal compositions, not vertical compositions. 
This does not matter for our purposes however, as we only work with the double nerve. 
The double nerve is the same in both definitions, except that the indices in the bisimplicial set are swapped.
}





\section{The discrete flow category}\label{sec:discrete_flow_category}

In this section we present the definition of the discrete flow category, as defined by Nanda in \cite{Nanda}. Note that we only define the discrete flow category of a discrete Morse function on a regular CW complex. For the general discrete flow category, see \cite{Nanda}.

\subsection{The entrance path category}\label{sec:ent}
We now describe an example of a p-category, the \emph{entrance path category} of a regular CW complex. It is similar to the face poset, but has more structure: it includes data on \emph{how} a cell is a face of another cell.

\begin{definition}
    Let $X$ be a regular CW complex. The \emph{entrance path category} $\Ent[X]$ of $X$ is a p-category given by the following.
    \begin{enumerate}
        \item The objects are the cells in $X$.
        \item The morphisms $x \to y$ are strictly descending sequences of cells $(x = x_0 > x_1 > \dots > x_k = y)$, with the understanding that $\id_x$ is the sequence $(x)$ with a single element.
        \item The partial order is defined so that $f \Rightarrow f'$ if and only if $f$ is a (not necessarily contiguous) subsequence of $f'$.
        \item Composition of morphisms are given by concatenating sequences as follows.
        \begin{equation*}
            \left(z > x_1 > \dots > x_k\right) \circ \left(y_0 > \dots > y_{l-1} > z\right)
            = \left(y_0 > \dots > y_{l-1} > z > x_1 > \dots x_k\right)
        \end{equation*}
    \end{enumerate}
\end{definition}

 \textcolor{black}{
A theorem, proved in \cite[Proposition 3.3]{Nanda}, states the following.
}

\begin{theorem}
    Let $X$ be a finite regular CW complex. Then $X$ is homotopy equivalent to $|\N(\Ent[X])|$.
\end{theorem}

\subsection{Localization}
Given a category $C$ and a collection of morphisms $Q$ that contains all identities and is closed under composition, one can construct a new category $C[Q^{-1}]$, called the \emph{localization} of $C$ at $Q$. This category is the minimal category containing $C$ where all morphisms in $Q$ have inverses. The localization comes with a functor $L \colon C \to C[Q^{-1}]$ that sends objects in $C$ to their copies in $C[Q^{-1}]$ and morphisms in $C$ to their equivalence classes in $C[Q^{-1}]$.

We now describe how to localize a p-category at a special class of morphisms.

\begin{definition}
    A morphism $f \colon x \to y$ in a p-category $C$ is called an \emph{atom} if
    \begin{enumerate}
        \item $f \Rightarrow f'$ holds for any $f' \in C(x, y)$,
        \item $x = y$ implies $f = \id_x$, and
        \item Solutions to $h \circ g \Rightarrow f$ for morphisms $g \colon x \to z$ and $h \colon z \to y$ only exist when \begin{itemize}
            \item $z = x$, in which case $(g, h) = (\id_x, f)$, or
            \item $z = y$, in which case $(g, h) = (f, \id_y)$.
        \end{itemize}
    \end{enumerate}
\end{definition}

As an example, in the entrance path category of a regular CW complex, the atoms are precisely the sequences of length 2, i.e., those on the form $(x > y)$, together with the identities.  \textcolor{black}{Hence, each non-identity morphism in the entrance path category has a unique decomposition into atoms. The atom decomposition of $(x_0 > x_1 > \dots > x_n)$ is simply $(x_n > x_{n-1}) \circ (x_{n-1} \circ x_{n-2}) \circ \cdots \circ (x_1 > x_0)$.}

\begin{definition}
    A collection $\Sigma$ of morphisms in a p-category is called \emph{directed} if
    \begin{enumerate}
        \item all morphisms in $\Sigma$ are atoms,
        \item if $f \colon x \to y$ is in $\Sigma$, then $x \ne y$, and
        \item if $f \colon x \to y$ is in $\Sigma$, then there are no morphisms $y \to x$ in $\Sigma$.
    \end{enumerate}
\end{definition}

 \textcolor{black}{We again give an example in the entrance path category $\Ent[X]$ of a regular CW complex $X$.} Let $f \colon X \to \R$ be a discrete Morse function. The gradient vector field $V_f = \{(x_i, y_i)\}$ gives a collection $\Sigma = \{(y_i > x_i)\}$ of morphisms in $\Ent[X]$. It's easily verified that $\Sigma$ is directed.

Now, given a p-category $C$ and a directed collection $\Sigma$ of morphisms in $C$ so that the union $\Sigma^{+}$ with all identities is closed under composition, we can define the \emph{localization of $C$ at $\Sigma$}. We write this as $\Loc_{\Sigma} C$, and define it as follows.
\begin{definition}
    Let $C$, $\Sigma$ and $\Sigma^{+}$ be as above. Then $\Loc_{\Sigma} C$ is a p-category given by the following.
    \begin{itemize}
        \item The objects are the same as the objects in $C$.
        \item The morphisms $w \to z$ are equivalence classes of zigzags of the form
        \begin{center}
        \begin{tikzpicture}
            \diagram{d}{3em}{3em}{
                w & y_0 & x_0 & y_1 & \cdots & x_k & z \\
            };
            
            \path[->,font = \scriptsize, midway]
            (d-1-1) edge node[above]{$g_0$} (d-1-2)
            (d-1-3) edge node[above]{$f_0$} (d-1-2)
            (d-1-3) edge node[above]{$g_1$} (d-1-4)
            (d-1-5) edge node[above]{$f_1$} (d-1-4)
            (d-1-6) edge node[above]{$f_k$} (d-1-5)
            (d-1-6) edge node[above]{$g_{k+1}$} (d-1-7);
        \end{tikzpicture}
        \end{center}
        where the $f_i$'s are in $\Sigma^{+}$, and the $g_i$'s are arbitrary, and the equivalence relation is generated by the following relations. Two zigzags are related \emph{horizontally} if they differ by intermediate identity maps, and \emph{vertically} if they form the rows of a commutative diagram of the form
        \begin{center}
        \begin{tikzpicture}
            \diagram{d}{3em}{3em}{
                w & y_0 & x_0 & y_1 & \cdots & x_k & z \\
                w & y_0' & x_0' & y_1' & \cdots & x_k' & z \\
            };
            
            \path[->,font = \scriptsize, midway]
            (d-1-1) edge node[above]{$g_0$} (d-1-2)
            (d-1-3) edge node[above]{$f_0$} (d-1-2)
            (d-1-3) edge node[above]{$g_1$} (d-1-4)
            (d-1-5) edge node[above]{$f_1$} (d-1-4)
            (d-1-6) edge node[above]{$f_k$} (d-1-5)
            (d-1-6) edge node[above]{$g_{k+1}$} (d-1-7);

            \path[->,font = \scriptsize, midway]
            (d-2-1) edge node[above]{$g_0'$} (d-2-2)
            (d-2-3) edge node[above]{$f_0'$} (d-2-2)
            (d-2-3) edge node[above]{$g_1'$} (d-2-4)
            (d-2-5) edge node[above]{$f_1'$} (d-2-4)
            (d-2-6) edge node[above]{$f_k'$} (d-2-5)
            (d-2-6) edge node[above]{$g_{k+1}'$} (d-2-7);

            \path[->,font = \scriptsize, midway]
            (d-1-1) edge node[left]{$\id_w$} (d-2-1)
            (d-1-2) edge node[left]{$u_0$} (d-2-2)
            (d-1-3) edge node[left]{$v_0$} (d-2-3)
            (d-1-4) edge node[left]{$u_1$} (d-2-4)
            (d-1-6) edge node[left]{$v_k$} (d-2-6)
            (d-1-7) edge node[left]{$\id_z$} (d-2-7);
        \end{tikzpicture}
        \end{center}
        where the $u_i$'s and $v_i$'s are in $\Sigma^{+}$.
        
        \item The partial order is defined so that $\gamma' \Rightarrow \gamma$ if and only if there exist representatives of $\gamma$ and $\gamma'$ that fit into the top and bottom row, respectively, of a (not necessarily commutative) diagram of the form
        \begin{center}
        \begin{tikzpicture}
            \diagram{d}{3em}{3em}{
                w & y_0 & x_0 & y_1 & \cdots & x_k & z \\
                w & y_0' & x_0' & y_1' & \cdots & x_k' & z \\
            };
            
            \path[->,font = \scriptsize, midway]
            (d-1-1) edge node[above]{$g_0$} (d-1-2)
            (d-1-3) edge node[above]{$f_0$} (d-1-2)
            (d-1-3) edge node[above]{$g_1$} (d-1-4)
            (d-1-5) edge node[above]{$f_1$} (d-1-4)
            (d-1-6) edge node[above]{$f_k$} (d-1-5)
            (d-1-6) edge node[above]{$g_{k+1}$} (d-1-7);

            \path[->,font = \scriptsize, midway]
            (d-2-1) edge node[above]{$g_0'$} (d-2-2)
            (d-2-3) edge node[above]{$f_0'$} (d-2-2)
            (d-2-3) edge node[above]{$g_1'$} (d-2-4)
            (d-2-5) edge node[above]{$f_1'$} (d-2-4)
            (d-2-6) edge node[above]{$f_k'$} (d-2-5)
            (d-2-6) edge node[above]{$g_{k+1}'$} (d-2-7);

            \path[->,font = \scriptsize, midway]
            (d-1-1) edge node[left]{$\id_w$} (d-2-1)
            (d-1-2) edge node[left]{$u_0$} (d-2-2)
            (d-1-3) edge node[left]{$v_0$} (d-2-3)
            (d-1-4) edge node[left]{$u_1$} (d-2-4)
            (d-1-6) edge node[left]{$v_k$} (d-2-6)
            (d-1-7) edge node[left]{$\id_z$} (d-2-7);

            \draw[->, double] (d-2-1) -- (d-1-2);
            \draw[->, double] (d-2-3) -- (d-1-2);
            \draw[->, double] (d-2-3) -- (d-1-4);
            \draw[->, double] (d-2-5) -- (d-1-4);
            \draw[->, double] (d-2-6) -- (d-1-5);
            \draw[->, double] (d-2-6) -- (d-1-7);
        \end{tikzpicture}
        \end{center}
        where, again, $u_i$ and $v_i$ are in $\Sigma^{+}$.
        
        \item Composition of morphisms are given by concatenating representatives as follows
        \begin{align*}
            & \left[w \xrightarrow{g_0} y_0 \xleftarrow{f_0} \dots \ \xleftarrow{f_k} x_k \xrightarrow{g_{k+1}} z\right]
            \circ \left[v \xrightarrow{g_0'} y_0' \xleftarrow{f_0'} \dots \ \xleftarrow{f_k'} x_k' \xrightarrow{g_{k+1}'} w\right] \\
            & \quad = \left[v \xrightarrow{g_0'} y_0' \xleftarrow{f_0'} \dots \ \xleftarrow{f_k'} x_k' \xrightarrow{g_0 \circ g_{k+1}'} y_0 \xleftarrow{f_0} \dots \ \xleftarrow{f_k} x_k \xrightarrow{g_{k+1}} z\right].
        \end{align*}
    \end{itemize}
\end{definition}

One also gets a functor (or equivalently, a strict p-functor) in this case, which we write $L_{\Sigma} \colon C \to \Loc_{\Sigma} C$. The functor sends objects to themselves and morphisms to their respective equivalence classes.

As mentioned above, a gradient vector field $V_f = \{(x_i, y_i)\}$ gives a directed collection $\Sigma = \{(y_i > x_i)\}$ on the entrance path category. Thus, given a discrete Morse function $f \colon X \to \R$, we can localize $\Ent[X]$ at $\Sigma$, which will give us a category where the morphisms corresponding to regular pairs have inverses. We will call this $\Sigma$ the \emph{Morse system} induced by $f$. 

\subsection{The discrete flow category}\label{sec:flo}

We are now ready to define the discrete flow category of a discrete Morse function.

\begin{definition}
    Let $f \colon X \to \R$ be a discrete Morse function, and let $\Sigma$ be its induced Morse system. Let $\Loc_{\Sigma}[X]$ be the p-category localization of $\Ent[X]$ at $\Sigma$. The \emph{discrete flow category} of $f$, denoted $\Flo_{\Sigma}[X]$, is the full subcategory of $\Loc_{\Sigma}[X]$ generated by the critical cells of $f$.
\end{definition}

A special case of the main result of \cite{Nanda} states the following.

\begin{theorem}
    Let $f \colon X \to \R$ be a discrete Morse function, and let $\Sigma$ be its induced Morse system. Then the classifying space of the discrete flow category of $f$ is homotopy equivalent to $X$. 
\end{theorem}

\begin{example}
As an example, consider the discrete Morse function on $X \cong S^1$ illustrated in \autoref{fig:circle_dmf}.

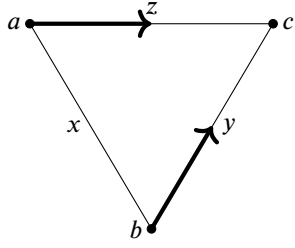
\begin{figure}[htbp]
    \centering
    \begin{tikzpicture}[scale=.8]
        \draw (4, 0) -- (2, 3.4) node[midway, anchor=east]{$x$};
        \draw (4, 0) -- (6, 3.4) node[midway, anchor=west]{$y$};
        \draw (2, 3.4) -- (6, 3.4) node[midway, anchor=south]{$z$};
        
        \filldraw (2, 3.4) circle (2 pt) node[anchor=east]{$a$};
        \filldraw (4, 0) circle (2 pt) node[anchor=east]{$b$};
        \filldraw (6, 3.4) circle (2 pt) node[anchor=west]{$c$};
        
        \draw[ultra thick, ->] (2, 3.4) -- (4, 3.4);
        \draw[ultra thick, ->] (4, 0) -- (5, 1.7);
    \end{tikzpicture}
    \caption{A gradient vector field of a discrete Morse function $f \colon X \to \R$.}
    \label{fig:circle_dmf}
\end{figure}

The objects of $\Flo_{\Sigma}[X]$ are the critical cells of $f$, which are $x$ and $c$. There are two morphisms from $x$ to $c$, corresponding to the two zigzag paths $x > b < y > c$ and $x > a < z > c$. There are no partial order relation between these morphisms. The other Hom posets are $\Hom(x, x) = \{\id_x\}$, $\Hom(c, c) = \{\id_c\}$ and $\Hom(c, x) = \emptyset$. Thus, the classifying space $|\N(\Flo_{\Sigma}[X])|$ becomes $S^1$, as illustrated in \autoref{fig:circle_dflow_nerve}.

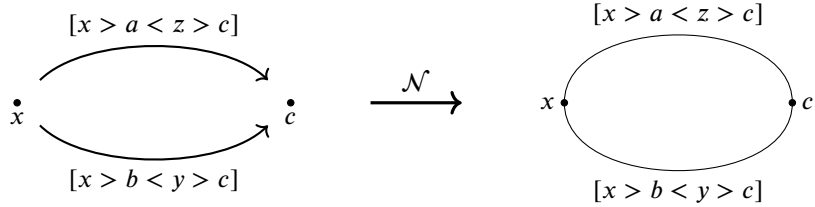
\begin{figure}[htbp]
    \centering
    \begin{tikzpicture}[scale=.6]
        \draw[thick, ->] (-1.5, 0.5) .. controls (-0.5, 1.5) and (2.5, 1.5) .. (3.5, 0.5) node[midway, anchor=south]{$[x > a < z > c]$};
        \draw[thick, ->] (-1.5, -0.5) .. controls (-0.5, -1.5) and (2.5, -1.5) .. (3.5, -0.5) node[midway, anchor=north]{$[x > b < y > c]$};
        
        \filldraw (-2, 0) circle (2 pt) node[anchor=north]{$x$};
        \filldraw (4, 0) circle (2 pt) node[anchor=north]{$c$};
        
        \draw[very thick, ->] (5.75, 0) -- (7.75, 0) node[midway, anchor=south]{$\N$};
        
        \draw (10, 0) .. controls (10, 2) and (15, 2) .. (15, 0) node[midway, anchor=south]{$[x > a < z > c]$};
        \draw (10, 0) .. controls (10, -2) and (15, -2) .. (15, 0) node[midway, anchor=north]{$[x > b < y > c]$};
        
        \filldraw (10, 0) circle (2 pt) node[anchor=east]{$x$};
        \filldraw (15, 0) circle (2 pt) node[anchor=west]{$c$};
    \end{tikzpicture}
    \caption{The discrete flow category and its nerve (the identity morphisms are omitted).}
    \label{fig:circle_dflow_nerve}
\end{figure}
\end{example}

\section{Hom posets of discrete flow categories}

In this section we study the structure of the Hom posets of the discrete flow category of a discrete Morse function.
We first develop some technical results on the Hom posets, and in particular show that the posets are \emph{graded}.
We then use these results to construct an algorithm to compute said posets.
Finally, we prove \autoref{theoremA}: for a discrete Morse function on a simplicial complex $X$, the opposite poset of each Hom poset in its discrete flow category is a \emph{CW poset}, i.e., it has the structure of a face poset of a regular CW complex.

\subsection{Preliminaries on posets}

In \cite{CW_posets}, Björner describes sufficient conditions for a poset to be the face poset of a regular CW complex. Note that he defines face posets to have an added least element $\hat 0$, as opposed to our definition of $\Fac[X]$. Björner's definition of CW posets is as follows.

\begin{definition}\label{def:cw_poset}
    A poset $P$ is said to be a \emph{CW poset} if
    \begin{enumerate}
        \item $P$ has a least element $\hat 0$,
        \item $P$ is nontrivial, i.e., has more than one element,
        \item For all $x \in P \setminus \{\hat 0\}$ the realization of the open interval $(\hat 0, x)$ is homeomorphic to a sphere (i.e., to some $S^k$, where $k$ depends on $x$).
    \end{enumerate}
\end{definition}

We now formalize the statement that CW posets are the face posets of regular CW complexes (augmented with a least element).

\begin{definition}
    Let $X$ be a regular CW complex. Let the \emph{face poset} of $X$, denoted $\Fac[X]$, be the poset where the elements are the cells in $X$ and $x \ge y$ whenever $y$ is contained in $x$.

    Denote by $\Fac^{+}[X]$ the face poset $\Fac[X]$ augmented with a least element $\hat 0$. In other words, $\Fac^{+}[X] = \Fac[X] \cup \{\hat 0\}$, where $\hat 0$ is defined to be smaller than all other elements.
\end{definition}

 \textcolor{black}{
The following statement is proved in \cite[Proposition 3.1]{CW_posets}.
}

\begin{theorem}
    A poset $P$ is a CW poset if and only if it is isomorphic to $\Fac^{+}[X]$ for some regular CW complex $X$.
\end{theorem}

%
%
%

A poset is \emph{bounded} if it has a least and greatest element. Furthermore, we say that $x$ \emph{covers} $y$ if $x > y$ and there exists no element $z$ such that $x > z > y$.

\begin{definition}\label{def:graded_poset}
    A graded poset is a poset $P$ equipped with a \emph{rank function}, which is a function $\rho \colon P \to \Z$ satisfying the following two properties.
    \begin{enumerate}
        \item The function $\rho$ is compatible with the partial order, meaning that $x > y$ implies $\rho(x) > \rho(y)$.
        \item The function $\rho$ is compatible with the covering relation, meaning that if $x$ covers $y$ then $\rho(x) = \rho(y) + 1$.
    \end{enumerate}
    For an element $x$, we will call $\rho(x)$ the \emph{rank} of $x$.
\end{definition}

\subsection{Hom posets of a discrete flow category}

In this section, we let $X$ be a finite regular CW complex, $f \colon X \to \R$ a discrete Morse function, and $\Sigma$ the Morse system on $\Ent[X]$ consisting of the regular pairs of $f$. Furthermore, we let $w$ and $z$ be arbitrary objects in $\Flo_{\Sigma}[X]$ and consider the Hom poset $\Hom(w, z)$ (note that whenever we write $\Hom(w, z)$, we mean $\Hom_{\Flo_{\Sigma}[X]}(w, z)$).

First, we construct an algebraic invariant for morphisms in the discrete flow category, which we will use when proving that two different morphisms are not equal.

Let $G$ be the free Abelian group on all morphisms in $\Ent[X]$ that are atoms, modulo all identities (i.e., we define identities to be 0). The algebraic invariant will be an element of this group. Given a representative 
\begin{equation*}
    \gamma = \left(w = x_0 \xrightarrow{g_0} y_0 \xleftarrow{f_0} x_1 \xrightarrow{g_1} \cdots \ \xleftarrow{f_{k-1}} x_k \xrightarrow{g_k} y_k = z\right)
\end{equation*}
of some $[\gamma] \in \Hom(w, z)$, we let
\begin{equation*}
    \alpha(g_i) = \left(g_i^0 + g_i^1 + \dots + g_i^{N_i}\right) \in G,
\end{equation*}
where $g_i = g_i^{N_i} \circ \dots \circ g_i^0$ is the atom decomposition of $g_i$. We now define the algebraic invariant $I$ as follows.

\begin{definition}\label{def:algebraic_invariant}
    Let $[\gamma] \in \Hom(w, z)$. Define $I([\gamma])$ as
    \begin{equation}
        I([\gamma]) = \sum_{i=0}^k \alpha(g_i) - \sum_{i=0}^{k-1} f_i.
    \end{equation}
\end{definition}

\begin{lemma}
    The function $I$, as defined in \autoref{def:algebraic_invariant}, is well-defined.
\end{lemma}

\begin{proof}
    To show that $I$ is well-defined, we must show that it is preserved under horizontal and vertical relations.
    \begin{description}[style=multiline]
        \item[(H)] To see that it is preserved under horizontal relations, it is enough to check that $I([\gamma \circ \gamma']) = I([\gamma]) + I([\gamma'])$, and that
        \begin{equation*}
            I([x_i \xrightarrow{g_i} y_i \xleftarrow{\id} x_{i+1} \xrightarrow{g_{i+1}} y_{i+1}]) = I([x_i \xrightarrow{g_{i+1} \circ g_i} y_{i+1}]),
        \end{equation*}
        which follows from the fact that identities are defined to be 0 in $G$, and that $\alpha(g \circ h) = \alpha(g) + \alpha(h)$.
        \item[(V)] To see that $I$ is preserved under vertical relations, consider a diagram:
        \begin{center}
        \begin{tikzpicture}
            \diagram{d}{3em}{3em}{
                w & y_0' & x_1' & y_1' & \cdots & x_k' & z \\
                w & y_0 & x_1 & y_1 & \cdots & x_k & z \\
            };
            
            \path[->,font = \scriptsize, midway]
            (d-1-1) edge node[above]{$g_0'$} (d-1-2)
            (d-1-3) edge node[above]{$f_0'$} (d-1-2)
            (d-1-3) edge node[above]{$g_1'$} (d-1-4)
            (d-1-5) edge node[above]{$f_1'$} (d-1-4)
            (d-1-6) edge node[above]{$f_{k-1}'$} (d-1-5)
            (d-1-6) edge node[above]{$g_k'$} (d-1-7);
    
            \path[->,font = \scriptsize, midway]
            (d-2-1) edge node[below]{$g_0$} (d-2-2)
            (d-2-3) edge node[below]{$f_0$} (d-2-2)
            (d-2-3) edge node[below]{$g_1$} (d-2-4)
            (d-2-5) edge node[below]{$f_1$} (d-2-4)
            (d-2-6) edge node[below]{$f_{k-1}$} (d-2-5)
            (d-2-6) edge node[below]{$g_k$} (d-2-7);
    
            \path[->,font = \scriptsize, midway]
            (d-1-2) edge node[left]{$u_0$} (d-2-2)
            (d-1-3) edge node[left]{$v_1$} (d-2-3)
            (d-1-4) edge node[left]{$u_1$} (d-2-4)
            (d-1-6) edge node[left]{$v_k$} (d-2-6);
            \draw[double] (d-1-1) -- (d-2-1);
            \draw[double] (d-1-7) -- (d-2-7);
    
            \draw[double] (d-2-1) -- (d-1-2);
            \draw[double] (d-2-3) -- (d-1-2);
            \draw[double] (d-2-3) -- (d-1-4);
            \draw[double] (d-2-5) -- (d-1-4);
            \draw[double] (d-2-6) -- (d-1-5);
            \draw[double] (d-2-6) -- (d-1-7);
        \end{tikzpicture}
        \end{center}

        We have $\alpha(g_i') = \alpha(g_i) + v_i - u_i$, and $f_i' = f_i + u_i - v_{i+1}$ (as elements of $G$). Putting this together gives
        \begin{align*}
            I([\gamma]) &= \sum_{i=0}^k \alpha(g_i') - \sum_{i=0}^{k-1} f_i' \\
            &= \sum_{i=0}^k \alpha(g_i) + \sum_{i=0}^k v_i - \sum_{i=0}^k u_i - \left(\sum_{i=0}^{k-1} f_i + \sum_{i=0}^{k-1} u_i - \sum_{i=0}^{k-1} v_{i+1} \right) \\
            &= \sum_{i=0}^k \alpha(g_i) - \sum_{i=0}^{k-1} f_i, \\
        \end{align*}
        where we use that $v_0 = \id_w = 0$ and $u_k = \id_z = 0$. Hence, $I$ is preserved under vertical relations.\qedhere
    \end{description}
\end{proof}


\begin{theorem}\label{thm:flo_po_desc}
    Let $[\gamma]$ and $[\gamma']$ be morphisms in $\Hom(w, z)$ such that $[\gamma] \Rightarrow [\gamma']$. Let
    \begin{equation*}
        \tau = \left(w = x_0 \xrightarrow{g_0} y_0 \xleftarrow{f_0} x_1 \xrightarrow{g_1} \cdots \ \xleftarrow{f_{k-1}} x_k \xrightarrow{g_k} y_k = z\right)
    \end{equation*}
    be a representative of $[\gamma]$.
    Then it is possible to choose a representative $\tau'$ of $[\gamma']$, such that
    \begin{equation*}
        \tau' = \left(w = x_0 \xrightarrow{g_0'} y_0 \xleftarrow{f_0} x_1 \xrightarrow{g_1'} \cdots \ \xleftarrow{f_{k-1}} x_k \xrightarrow{g_k'} y_k = z\right),
    \end{equation*}
    for some $g_i'$, such that $g_i \Rightarrow g_i'$ holds for all $i$.
    
    Furthermore, the $g_i'$ are unique.
\end{theorem}

Note that a partial order $g_i \Rightarrow g_i'$ means that $g_i = (z_0 > z_1 > \dots > z_m)$ is a subsequence of $g_i' = (z_0' > z_1' > \dots > z_n')$. Therefore, this theorem tells us that a partial order $[\gamma] \Rightarrow [\gamma']$ corresponds to picking a representative of $[\gamma]$ and adding elements to the sequences that constitutes the right-pointing arrows (the $g_i$).


\begin{proof}
    We prove this in two parts. First we show that there exists a diagram

    \begin{equation}\label{eq:po_diagram}\begin{aligned}
    \begin{tikzpicture}
        \diagram{d}{3em}{3em}{
            w & \bar y_0 & \bar x_1 & \bar y_1 & \cdots & \bar x_k & z \\
            w & y_0 & x_1 & y_1 & \cdots & x_k & z \\
        };
        
        \path[->,font = \scriptsize, midway]
        (d-1-1) edge node[above]{$\bar g_0$} (d-1-2)
        (d-1-3) edge node[above]{$\bar f_0$} (d-1-2)
        (d-1-3) edge node[above]{$\bar g_1$} (d-1-4)
        (d-1-5) edge node[above]{$\bar f_1$} (d-1-4)
        (d-1-6) edge node[above]{$\bar f_{k-1}$} (d-1-5)
        (d-1-6) edge node[above]{$\bar g_k$} (d-1-7);

        \path[->,font = \scriptsize, midway]
        (d-2-1) edge node[below]{$g_0$} (d-2-2)
        (d-2-3) edge node[below]{$f_0$} (d-2-2)
        (d-2-3) edge node[below]{$g_1$} (d-2-4)
        (d-2-5) edge node[below]{$f_1$} (d-2-4)
        (d-2-6) edge node[below]{$f_{k-1}$} (d-2-5)
        (d-2-6) edge node[below]{$g_k$} (d-2-7);

        \path[->,font = \scriptsize, midway]
        (d-1-2) edge node[left]{$u_0$} (d-2-2)
        (d-1-3) edge node[left]{$v_1$} (d-2-3)
        (d-1-4) edge node[left]{$u_1$} (d-2-4)
        (d-1-6) edge node[left]{$v_k$} (d-2-6);
        \draw[double] (d-1-1) -- (d-2-1);
        \draw[double] (d-1-7) -- (d-2-7);

        \draw[->, double] (d-2-1) -- (d-1-2);
        \draw[->, double] (d-2-3) -- (d-1-2);
        \draw[->, double] (d-2-3) -- (d-1-4);
        \draw[->, double] (d-2-5) -- (d-1-4);
        \draw[->, double] (d-2-6) -- (d-1-5);
        \draw[->, double] (d-2-6) -- (d-1-7);
    \end{tikzpicture}
    \end{aligned}\end{equation}
    where the bottom row is $\tau$ and the top row is a representative of $[\gamma']$. Then we show that we can modify this diagram so that the top row is of the form
    \begin{equation*}
        \tau' = \left(w = x_0 \xrightarrow{g_0'} y_0 \xleftarrow{f_0} x_1 \xrightarrow{g_1'} \cdots \ \xleftarrow{f_{k-1}} x_k \xrightarrow{g_k'} y_k = z\right),
    \end{equation*}
    where $\tau'$ is still a representative of $[\gamma']$.

    To show the first part, we show that given a diagram representing a partial order relation (as the one in \eqref{eq:po_diagram}), we can replace the bottom row by either
    \begin{description}[style=multiline]
        \item[(V)] a vertically related representative, or
        \item[(H)] a horizontally related representative,
    \end{description}
    and modify the top row appropriately to get a new diagram.
    \begin{description}[style=multiline]
        \item[(V)] The case with vertically related representatives is simple; you simply compose the partial order diagram with the vertical relation diagram.
        
        \item[(H)] For horizontal relations, we must show that you can both add intermediate identity maps, and remove them.
        
        The case with adding intermediate identity maps is also simple, you simply add identity maps to both the top and bottom rows and copy the vertical map (clearly, the top rows is also horizontally related). For example, for adding an identity map at an $x_i$, you replace

        \begin{center}
        \begin{tikzpicture}[scale=.5]
            \diagram{d}{3em}{3em}{
                \cdots & x_i' & \cdots \\
                \cdots & x_i & \cdots \\
            };
            
            \path[->,font = \scriptsize, midway]
            (d-1-2) edge (d-1-1)
            (d-1-2) edge (d-1-3);
    
            \path[->,font = \scriptsize, midway]
            (d-2-2) edge (d-2-1)
            (d-2-2) edge (d-2-3);
    
            \path[->,font = \scriptsize, midway]
            (d-1-2) edge node[left]{$v_i$} (d-2-2);
    
            \draw[->, double] (d-2-2) -- (d-1-1);
            \draw[->, double] (d-2-2) -- (d-1-3);
        \end{tikzpicture}
        \end{center}

        with 

        \begin{center}
        \begin{tikzpicture}[scale=.5]
            \diagram{d}{3em}{3em}{
                \cdots & x_i' & x_i' & x_i' & \cdots \\
                \cdots & x_i & x_i & x_i & \cdots \\
            };
            
            \path[->,font = \scriptsize, midway]
            (d-1-2) edge (d-1-1)
            (d-1-2) edge node[above]{$\id$} (d-1-3)
            (d-1-4) edge node[above]{$\id$} (d-1-3)
            (d-1-4) edge (d-1-5);
    
            \path[->,font = \scriptsize, midway]
            (d-2-2) edge (d-2-1)
            (d-2-2) edge node[below]{$\id$} (d-2-3)
            (d-2-4) edge node[below]{$\id$} (d-2-3)
            (d-2-4) edge (d-2-5);
    
            \path[->,font = \scriptsize, midway]
            (d-1-2) edge node[left]{$v_i$} (d-2-2)
            (d-1-3) edge node[left]{$v_i$} (d-2-3)
            (d-1-4) edge node[left]{$v_i$} (d-2-4);
    
            \draw[->, double] (d-2-2) -- (d-1-1);
            \draw[->, double] (d-2-2) -- (d-1-3);
            \draw[->, double] (d-2-4) -- (d-1-3);
            \draw[->, double] (d-2-4) -- (d-1-5);
        \end{tikzpicture}
        \end{center}

        Removing intermediate identity maps is more complicated (in fact, it only works on the bottom row, not the top row). Consider a diagram

        \begin{center}
        \begin{tikzpicture}
            \diagram{d}{3em}{3em}{
                x_i' & y_i' & x_{i+1}' & y_{i+1}' \\
                x_i & y_i & x_{i+1} & y_{i+1} \\
            };
            
            \path[->,font = \scriptsize, midway]
            (d-1-1) edge node[above]{$g_i'$} (d-1-2)
            (d-1-3) edge node[above]{$f_i'$} (d-1-2)
            (d-1-3) edge node[above]{$g_{i+1}'$} (d-1-4);
    
            \path[->,font = \scriptsize, midway]
            (d-2-1) edge node[below]{$g_i$} (d-2-2)
            (d-2-3) edge node[below]{$g_{i+1}$} (d-2-4);
            \draw[double] (d-2-2) -- (d-2-3);
    
            \path[->,font = \scriptsize, midway]
            (d-1-1) edge node[left]{$v_i$} (d-2-1)
            (d-1-2) edge node[left]{$u_i$} (d-2-2)
            (d-1-3) edge node[left]{$v_{i+1}$} (d-2-3)
            (d-1-4) edge node[right]{$u_{i+1}$} (d-2-4);
    
            \draw[->, double] (d-2-1) -- (d-1-2);
            \draw[->, double] (d-2-3) -- (d-1-2);
            \draw[->, double] (d-2-3) -- (d-1-4);
        \end{tikzpicture}
        \end{center}

        If $f_i' = \id$, we can clearly remove the identity maps from both rows. Suppose not. Then $u_i$ must be $\id$ and $v_{i+1} = f_i'$. As $g_{i+1} \circ v_{i+1} \Rightarrow u_{i+1} \circ g_{i+1}'$, and $v_{i+1}$ is an atom on the form $(x_{i+1}' > x_{i+1})$ with $\dim x_{i+1}' = \dim x_{i+1} + 1$, we must have that $v_{i+1} \circ g_{i+1}'$ is sequence that starts with $(x_1' > x_1)$. There are now two cases: either $g_{i+1}' = \hat g_{i+1} \circ v_{i+1}$ for some $\hat g_{i+1}$, or $g_{i+1}' = \id$ and $u_{i+1} = v_{i+1}$.

        In the first case, the top row is equivalent to
        \begin{equation*}
            x_i' \xrightarrow{\hat g_{i+1} \circ g_i'} y_{i+1},
        \end{equation*}
        as $g_{i+1}' = \hat g_i \circ v_{i+1} = \hat g_i \circ f_i'$, and we can replace the diagram with

        \begin{center}
        \begin{tikzpicture}
            \diagram{d}{3em}{3em}{
                x_i' & y_{i+1}' \\
                x_i & y_{i+1} \\
            };
            
            \path[->,font = \scriptsize, midway]
            (d-1-1) edge node[above]{$\hat g_{i+1} \circ g_i'$} (d-1-2);
    
            \path[->,font = \scriptsize, midway]
            (d-2-1) edge node[below]{$g_{i+1} \circ g_i$} (d-2-2);
    
            \path[->,font = \scriptsize, midway]
            (d-1-1) edge node[left]{$v_i$} (d-2-1)
            (d-1-2) edge node[right]{$u_{i+1}$} (d-2-2);
    
            \draw[->, double] (d-2-1) -- (d-1-2);
        \end{tikzpicture}
        \end{center}

        The reader can verify that in the case $g_{i+1}' = \id$, $f_{i+1}'$ will be forced to equal $\id$, which allows us to make a similar rewrite to remove the identity map from the bottom row.
    \end{description}
    
    Now, we have a diagram

    \begin{center}
    \begin{tikzpicture}
        \diagram{d}{3em}{3em}{
            w & y_0' & x_1' & y_1' & \cdots & x_k' & z \\
            w & y_0 & x_1 & y_1 & \cdots & x_k & z \\
        };
        
        \path[->,font = \scriptsize, midway]
        (d-1-1) edge node[above]{$g_0'$} (d-1-2)
        (d-1-3) edge node[above]{$f_0'$} (d-1-2)
        (d-1-3) edge node[above]{$g_1'$} (d-1-4)
        (d-1-5) edge node[above]{$f_1'$} (d-1-4)
        (d-1-6) edge node[above]{$f_{k-1}'$} (d-1-5)
        (d-1-6) edge node[above]{$g_k'$} (d-1-7);

        \path[->,font = \scriptsize, midway]
        (d-2-1) edge node[below]{$g_0$} (d-2-2)
        (d-2-3) edge node[below]{$f_0$} (d-2-2)
        (d-2-3) edge node[below]{$g_1$} (d-2-4)
        (d-2-5) edge node[below]{$f_1$} (d-2-4)
        (d-2-6) edge node[below]{$f_{k-1}$} (d-2-5)
        (d-2-6) edge node[below]{$g_k$} (d-2-7);

        \path[->,font = \scriptsize, midway]
        (d-1-2) edge node[left]{$u_0$} (d-2-2)
        (d-1-3) edge node[left]{$v_1$} (d-2-3)
        (d-1-4) edge node[left]{$u_1$} (d-2-4)
        (d-1-6) edge node[left]{$v_k$} (d-2-6);
        \draw[double] (d-1-1) -- (d-2-1);
        \draw[double] (d-1-7) -- (d-2-7);

        \draw[->, double] (d-2-1) -- (d-1-2);
        \draw[->, double] (d-2-3) -- (d-1-2);
        \draw[->, double] (d-2-3) -- (d-1-4);
        \draw[->, double] (d-2-5) -- (d-1-4);
        \draw[->, double] (d-2-6) -- (d-1-5);
        \draw[->, double] (d-2-6) -- (d-1-7);
    \end{tikzpicture}
    \end{center}
    
    such that the bottom row is a $[\tau]$ and the top row is a representative of $[\gamma']$.

    Assume either $x_1 \ne x_1'$ or $y_0 \ne y_0'$. There are three cases to consider:
    \begin{description}[style=multiline]
    \item[(I)] $x_1 \ne x_1'$ and $y_0 \ne y_0'$
    \item[(II)] $x_1 = x_1'$ and $y_0\ne y_0'$
    \item[(III)] $x_1 \ne x_1'$ and $y_0 = y_0'$.
    \end{description}

	 \textcolor{black}{In the following three cases, we will use that $\Sigma$ satisfies the \emph{exhaustion axiom} (\cite[Definition 3.5]{Nanda}), which implies that if $f \colon x \to y$ is in $\Sigma$, then no other morphism in $\Sigma$ has either $x$ or $y$ as either source or target. It is straightforward to show that $\Sigma$ has this property, as regular pairs of a discrete Morse function are pairwise disjoint. }

    \begin{description}[labelwidth=1.25cm,leftmargin=!]
        \item[Case I] 
        In this case, the exhaustion axiom gives that $y_0 = x_1$ and $y_0' = x_1'$, which implies that $v_1 = u_0$. The first three squares in the diagram is then:

        \begin{center}
        \begin{tikzpicture}
            \diagram{d}{3em}{3em}{
                w & y_0' & x_1' & y_1' \\
                w & y_0 & x_1 & y_1 \\
            };
            
            \path[->,font = \scriptsize, midway]
            (d-1-1) edge node[above]{$g_0'$} (d-1-2)
            (d-1-3) edge node[above]{$g_1'$} (d-1-4);
            \draw[double] (d-1-2) -- (d-1-3);
    
            \path[->,font = \scriptsize, midway]
            (d-2-1) edge node[below]{$g_0$} (d-2-2)
            (d-2-3) edge node[below]{$g_1$} (d-2-4);
            \draw[double] (d-2-2) -- (d-2-3);
    
            \path[->,font = \scriptsize, midway]
            (d-1-2) edge node[left]{$u_0$} (d-2-2)
            (d-1-3) edge node[left]{$u_0$} (d-2-3)
            (d-1-4) edge node[left]{$u_1$} (d-2-4);
            \draw[double] (d-1-1) -- (d-2-1);
    
            \draw[->, double] (d-2-1) -- (d-1-2);
            \draw[->, double] (d-2-3) -- (d-1-2);
            \draw[->, double] (d-2-3) -- (d-1-4);
        \end{tikzpicture}
        \end{center}

        As $g_1 \circ u_0 \Rightarrow u_1 \circ g_1'$, and $u_0$ is an atom on the form $(x_1' > x_1)$ with $\dim x_1' = \dim x_1 + 1$, we must have that $u_1 \circ g_1'$ is sequence that starts with $(x_1' > x_1)$. There are now two cases: either $g_1' = \hat g_1 \circ u_0$ for some $\hat g_1$, or $g_1' = \id_{x_1'}$ and $u_1 = u_0 = g_1$.
        
        In the first case, the top row is vertically related to 
        \begin{equation*}
            w \xrightarrow{u_0 \circ g_0'} y_0 \xleftarrow{\id} x_1 \xrightarrow{\hat g_1} y_1',
        \end{equation*}
        so we can replace the diagram with:

        \begin{center}
        \begin{tikzpicture}
            \diagram{d}{3em}{3em}{
                w & y_0 & x_1 & y_1' \\
                w & y_0 & x_1 & y_1 \\
            };
            
            \path[->,font = \scriptsize, midway]
            (d-1-1) edge node[above]{$u_0 \circ g_0'$} (d-1-2)
            (d-1-3) edge node[above]{$\hat g_1$} (d-1-4);
            \draw[double] (d-1-2) -- (d-1-3);
    
            \path[->,font = \scriptsize, midway]
            (d-2-1) edge node[below]{$g_0$} (d-2-2)
            (d-2-3) edge node[below]{$g_1$} (d-2-4);
            \draw[double] (d-2-2) -- (d-2-3);
    
            \path[->,font = \scriptsize, midway]
            (d-1-4) edge node[left]{$u_1$} (d-2-4);
            \draw[double] (d-1-1) -- (d-2-1);
            \draw[double] (d-1-2) -- (d-2-2);
            \draw[double] (d-1-3) -- (d-2-3);
    
            \draw[->, double] (d-2-1) -- (d-1-2);
            \draw[->, double] (d-2-3) -- (d-1-2);
            \draw[->, double] (d-2-3) -- (d-1-4);
        \end{tikzpicture}
        \end{center}

        In the second case, we must have $g_1 = \id$ and $u_1 = \id$. Hence, the first four squares in the diagram:

        \begin{center}
        \begin{tikzpicture}
            \diagram{d}{3em}{3em}{
                w & y_0' & x_1' & y_1' & x_2' \\
                w & y_0 & x_1 & y_1 & x_2 \\
            };
            
            \path[->,font = \scriptsize, midway]
            (d-1-1) edge node[above]{$g_0'$} (d-1-2);
            \draw[double] (d-1-2) -- (d-1-3);
            \draw[double] (d-1-3) -- (d-1-4);
            \draw[double] (d-1-4) -- (d-1-5);
    
            \path[->,font = \scriptsize, midway]
            (d-2-1) edge node[below]{$g_0$} (d-2-2)
            (d-2-5) edge node[below]{$f_1$} (d-2-4);
            \draw[double] (d-2-2) -- (d-2-3);
            \draw[double] (d-2-3) -- (d-2-4);
    
            \path[->,font = \scriptsize, midway]
            (d-1-2) edge node[left]{$u_0$} (d-2-2)
            (d-1-3) edge node[left]{$u_0$} (d-2-3)
            (d-1-4) edge node[left]{$u_0$} (d-2-4)
            (d-1-5) edge node[left]{$v_2$} (d-2-5);
            \draw[double] (d-1-1) -- (d-2-1);
    
            \draw[->, double] (d-2-1) -- (d-1-2);
            \draw[->, double] (d-2-3) -- (d-1-2);
            \draw[->, double] (d-2-3) -- (d-1-4);
            \draw[->, double] (d-2-5) -- (d-1-4);
        \end{tikzpicture}
        \end{center}

        Now, the top row is vertically related to
        \begin{equation*}
            w \xrightarrow{u_0 \circ g_0'} y_0 \xleftarrow{\id} x_1 \xrightarrow{\id} y_1' \xleftarrow{u_0} x_2',
        \end{equation*}
        so we can replace the diagram with:

        \begin{center}
        \begin{tikzpicture}
            \diagram{d}{3em}{3em}{
                w & y_0' & x_1' & y_1' & x_2' \\
                w & y_0 & x_1 & y_1 & x_2 \\
            };
            
            \path[->,font = \scriptsize, midway]
            (d-1-1) edge node[above]{$u_0 \circ g_0'$} (d-1-2)
            (d-1-5) edge node[above]{$u_0$} (d-1-4);
            \draw[double] (d-1-2) -- (d-1-3);
            \draw[double] (d-1-3) -- (d-1-4);
    
            \path[->,font = \scriptsize, midway]
            (d-2-1) edge node[below]{$g_0$} (d-2-2)
            (d-2-5) edge node[below]{$f_1$} (d-2-4);
            \draw[double] (d-2-2) -- (d-2-3);
            \draw[double] (d-2-3) -- (d-2-4);
    
            \path[->,font = \scriptsize, midway]
            (d-1-5) edge node[left]{$v_2$} (d-2-5);
            \draw[double] (d-1-1) -- (d-2-1);
            \draw[double] (d-1-2) -- (d-2-2);
            \draw[double] (d-1-3) -- (d-2-3);
            \draw[double] (d-1-4) -- (d-2-4);
    
            \draw[->, double] (d-2-1) -- (d-1-2);
            \draw[->, double] (d-2-3) -- (d-1-2);
            \draw[->, double] (d-2-3) -- (d-1-4);
            \draw[->, double] (d-2-5) -- (d-1-4);
        \end{tikzpicture}
        \end{center}

        %
        %
    %
        %
    %
        
        \item[Case II] 
        In this case the exhaustion axiom gives $y_0' = x_1'$ and $f_0 = u_0$, so the first three squares of the diagram is:

        \begin{center}
        \begin{tikzpicture}
            \diagram{d}{3em}{3em}{
                w & y_0' & x_1' & y_1' \\
                w & y_0 & x_1 & y_1 \\
            };
            
            \path[->,font = \scriptsize, midway]
            (d-1-1) edge node[above]{$g_0'$} (d-1-2)
            (d-1-3) edge node[above]{$g_1'$} (d-1-4);
            \draw[double] (d-1-2) -- (d-1-3);
    
            \path[->,font = \scriptsize, midway]
            (d-2-1) edge node[below]{$g_0$} (d-2-2)
            (d-2-3) edge node[below]{$u_0$} (d-2-2)
            (d-2-3) edge node[below]{$g_1$} (d-2-4);
    
            \path[->,font = \scriptsize, midway]
            (d-1-2) edge node[left]{$u_0$} (d-2-2)
            (d-1-4) edge node[left]{$u_1$} (d-2-4);
            \draw[double] (d-1-1) -- (d-2-1);
            \draw[double] (d-1-3) -- (d-2-3);
    
            \draw[->, double] (d-2-1) -- (d-1-2);
            \draw[->, double] (d-2-3) -- (d-1-2);
            \draw[->, double] (d-2-3) -- (d-1-4);
        \end{tikzpicture}
        \end{center}

        The top row is vertically related to $x_0 \xrightarrow{u_0 \circ g_0'} y_0 \xleftarrow{u_0} x_1 \xrightarrow{g_1} y_1'$ (see \cite[Remark 2.9]{Nanda}). We can thus replace this part of the diagram with:

        \begin{center}
        \begin{tikzpicture}
            \diagram{d}{3em}{3em}{
                w & y_0 & x_1' & y_1' \\
                w & y_0 & x_1 & y_1 \\
            };
            
            \path[->,font = \scriptsize, midway]
            (d-1-1) edge node[above]{$u_0 \circ g_0'$} (d-1-2)
            (d-1-3) edge node[above]{$u_0$} (d-1-2)
            (d-1-3) edge node[above]{$g_1'$} (d-1-4);
    
            \path[->,font = \scriptsize, midway]
            (d-2-1) edge node[below]{$g_0$} (d-2-2)
            (d-2-3) edge node[below]{$u_0$} (d-2-2)
            (d-2-3) edge node[below]{$g_1$} (d-2-4);
    
            \path[->,font = \scriptsize, midway]
            (d-1-4) edge node[left]{$u_1$} (d-2-4);
            \draw[double] (d-1-1) -- (d-2-1);
            \draw[double] (d-1-2) -- (d-2-2);
            \draw[double] (d-1-3) -- (d-2-3);
    
            \draw[->, double] (d-2-1) -- (d-1-2);
            \draw[->, double] (d-2-3) -- (d-1-2);
            \draw[->, double] (d-2-3) -- (d-1-4);
        \end{tikzpicture}
        \end{center}

        \item[Case III] 
        In this case the exhaustion axiom gives $y_0 = x_1$ and $f_0' = v_1$, so the first three squares of the diagram is:

        \begin{center}
        \begin{tikzpicture}
            \diagram{d}{3em}{3em}{
                w & y_0' & x_1' & y_1' \\
                w & y_0 & x_1 & y_1 \\
            };
            
            \path[->,font = \scriptsize, midway]
            (d-1-1) edge node[above]{$g_0'$} (d-1-2)
            (d-1-3) edge node[above]{$v_1$} (d-1-2)
            (d-1-3) edge node[above]{$g_1'$} (d-1-4);
    
            \path[->,font = \scriptsize, midway]
            (d-2-1) edge node[below]{$g_0$} (d-2-2)
            (d-2-3) edge node[below]{$g_1$} (d-2-4);
            \draw[double] (d-2-2) -- (d-2-3);
    
            \path[->,font = \scriptsize, midway]
            (d-1-3) edge node[left]{$v_1$} (d-2-3)
            (d-1-4) edge node[left]{$u_1$} (d-2-4);
            \draw[double] (d-1-1) -- (d-2-1);
            \draw[double] (d-1-2) -- (d-2-2);
    
            \draw[->, double] (d-2-1) -- (d-1-2);
            \draw[->, double] (d-2-3) -- (d-1-2);
            \draw[->, double] (d-2-3) -- (d-1-4);
        \end{tikzpicture}
        \end{center}

        As in the first case, either $g_1' = \id$ or $g_1' = \hat g_1 \circ v_1$ for some $\hat g_1$. In the latter case, the top row is vertically related to
        \begin{equation*}
            w \xrightarrow{g_0'} y_0' \xleftarrow{\id} x_1 \xrightarrow{\hat g_1} y_1',
        \end{equation*}
        and we can replace the diagram with

        \begin{center}
        \begin{tikzpicture}
            \diagram{d}{3em}{3em}{
                w & y_0 & x_1 & y_1' \\
                w & y_0 & x_1 & y_1 \\
            };
            
            \path[->,font = \scriptsize, midway]
            (d-1-1) edge node[above]{$g_0'$} (d-1-2)
            (d-1-3) edge node[above]{$\hat g_1$} (d-1-4);
            \draw[double] (d-1-2) -- (d-1-3);
    
            \path[->,font = \scriptsize, midway]
            (d-2-1) edge node[below]{$g_0$} (d-2-2)
            (d-2-3) edge node[below]{$g_1$} (d-2-4);
            \draw[double] (d-2-2) -- (d-2-3);
    
            \path[->,font = \scriptsize, midway]
            (d-1-4) edge node[left]{$u_1$} (d-2-4);
            \draw[double] (d-1-1) -- (d-2-1);
            \draw[double] (d-1-2) -- (d-2-2);
            \draw[double] (d-1-3) -- (d-2-3);
    
            \draw[->, double] (d-2-1) -- (d-1-2);
            \draw[->, double] (d-2-3) -- (d-1-2);
            \draw[->, double] (d-2-3) -- (d-1-4);
        \end{tikzpicture}
        \end{center}

        If, instead, $g_1' = \id$, we must have $g_1 = \id$ and $u_1 = v_1$. As in case 1, we look at the first four squares of the diagram, and see that we can replace the top row, in this case with
        \begin{equation*}
            w \xrightarrow{g_0'} y_0' \xleftarrow{\id} x_1 \xrightarrow{\id} y_1 \xleftarrow{v_1} x_2'.
        \end{equation*}

        %
        %
    %
    %
    %

    \end{description}

    Continuing this process on the right until you reach the end of the diagram, you end up with a new diagram on the form

    \begin{center}
    \begin{tikzpicture}
        \diagram{d}{3em}{3em}{
            w & y_0 & x_1 & y_1 & \cdots & x_k & z \\
            w & y_0 & x_1 & y_1 & \cdots & x_k & z \\
        };
        
        \path[->,font = \scriptsize, midway]
        (d-1-1) edge node[above]{$g_0''$} (d-1-2)
        (d-1-3) edge node[above]{$f_0$} (d-1-2)
        (d-1-3) edge node[above]{$g_1''$} (d-1-4)
        (d-1-5) edge node[above]{$f_1$} (d-1-4)
        (d-1-6) edge node[above]{$f_{k-1}$} (d-1-5)
        (d-1-6) edge node[above]{$g_k''$} (d-1-7);

        \path[->,font = \scriptsize, midway]
        (d-2-1) edge node[below]{$g_0$} (d-2-2)
        (d-2-3) edge node[below]{$f_0$} (d-2-2)
        (d-2-3) edge node[below]{$g_1$} (d-2-4)
        (d-2-5) edge node[below]{$f_1$} (d-2-4)
        (d-2-6) edge node[below]{$f_{k-1}$} (d-2-5)
        (d-2-6) edge node[below]{$g_k$} (d-2-7);

        \draw[double] (d-1-1) -- (d-2-1);
        \draw[double] (d-1-2) -- (d-2-2);
        \draw[double] (d-1-3) -- (d-2-3);
        \draw[double] (d-1-4) -- (d-2-4);
        \draw[double] (d-1-6) -- (d-2-6);
        \draw[double] (d-1-7) -- (d-2-7);

        \draw[->, double] (d-2-1) -- (d-1-2);
        \draw[->, double] (d-2-3) -- (d-1-2);
        \draw[->, double] (d-2-3) -- (d-1-4);
        \draw[->, double] (d-2-5) -- (d-1-4);
        \draw[->, double] (d-2-6) -- (d-1-5);
        \draw[->, double] (d-2-6) -- (d-1-7);
    \end{tikzpicture}
    \end{center}

    where  \textcolor{black}{the bottom row} is $\tau$ and the top row is a representative of $[\gamma']$. It follows from the diagram that $g_i \Rightarrow g_i''$ for all $i$, and hence the top row is our desired representative $\tau'$.

    Finally, we prove that the $g_i'$ are unique. We do this by using the algebraic invariant $I$, as defined in \autoref{def:algebraic_invariant}, and showing that given $I$, the $g_i'$ are decided by the $x_i$, $y_i$ and $f_i$.

    We know that $I$ is an algebraic invariant for morphisms in $\Hom(w, z)$. Thus, given a representation $\tau'$ as in the theorem statement, where the $x_i$, $y_i$ and $f_i$ are given, we know what all the atoms in the $g_i'$ are. It remains to show that these atoms can only appear in one specific order, which will mean that all the $g_i'$ are uniquely defined.
    
    
    For this, recall that $f \colon X \to \R$ is the Morse function from which the Morse system is induced. We can assume without loss of generality that $f$ is injective on the cells of $X$. Now, for an atom $(x > y)$, let $\beta((x>y)) = (\dim x, f(x), f(y)) \in \R^3$. Then $\beta((x>y)) = \beta((x'>y'))$ if and only if $(x>y) = (x'>y')$ (as $f$ is injective). Furthermore, using the lexicographical ordering on $\R^3$, $\beta$ is decreasing on the atoms in the right-pointing arrows of a representation, in the order they appear:
    \begin{itemize}
        \item If $(x' > y')$ directly succeeds $(x > y)$ (possibly with an intermediate identity maps), then $x' = y$, and $\dim x' > \dim y' = \dim x > \dim y$.
        \item If there is a non-identity left-pointing map $f_i$ directly between two atoms $(x>y)$ and $(x'>y')$, then $f_i = (x'>y)$, and $\dim x = \dim x'$. Then, we have $f(y) > f(x')$, and either $x = x'$ or $f(x) > f(y)$, so $f(x) \ge f(x')$. Likewise, either $y = y'$ or $f(x') > f(y')$ so $f(y) \ge f(y')$. In conclusion, this gives $\beta((x>y)) \ge \beta((x'>y'))$.
    \end{itemize}

    Thus, given the $x_i$, $y_i$ and $f_i$ in a representation $\tau'$ as in the theorem statement, the atoms in the $g_i'$, and their order, is uniquely determined, and hence the $g_i'$ themselves are uniquely determined, which proves the last part of the theorem.
\end{proof}


Now, let $\Hom(w, z)^{\op}$ be the \emph{opposite} of the poset $\Hom(w, z)$, i.e., the poset with elements equal to the elements of $\Hom(w, z)$ and with the partial order reversed. Define further $P_{w, z}$ as $\Hom(w, z)^{\op}$ augmented with a least element $\hat 0$. In other words,
\begin{equation}\label{eq:Pwz_def}
    P_{w, z} = \Hom(w, z)^{\op} \cup \{\hat 0\},
\end{equation}
where $\hat 0$ is defined to be smaller than all other elements.

Next, we show that $P_{w, z}$ is graded by defining a rank function on it. To do this, we first define the following rank function $r$ on $\Hom_{\Ent[X]}(a, b)^{\op}$, for arbitrary objects $a, b \in \Ent[X]$.
\begin{equation}\label{eq:r_def}
    r\left(a = x_0 > x_1 > \dots > x_k = b\right) = k
\end{equation}
Here, we take $r(\id_x)$ to be $0$.

It's easy to verify that $r$ is, in fact, a rank function. Also, observe that $r$ satisfies the following identity.

\begin{equation}\label{eq:r_concat}
    r(\gamma \circ \gamma') = r(\gamma) + r(\gamma')
\end{equation}

We now define a rank function on $P_{w, z}$.

\begin{definition}\label{def:P_rank}
    Define the function $\rho \colon P_{w, z} \to \mathbb{Z}$ as follows.
    
    Let $x \in P_{w, z}$. If $x = \hat 0$, then $\rho(x) = -1$. Else, let $x$ be represented by
    \begin{equation*}
        \tau = \left(w = x_0 \xrightarrow{g_0} y_0 \xleftarrow{f_0} x_1 \xrightarrow{g_1} \cdots \ \xleftarrow{f_{k-1}} x_k \xrightarrow{g_k} y_k = z\right).
    \end{equation*}
    Then,
    \begin{equation*}
        \rho(x) = N + \sum_{i=0}^{k-1} r(f_i) - \sum_{i=0}^{k} r(g_i),
    \end{equation*}
    where $N = \dim w - \dim z$.
\end{definition}

The rank function $\rho$ can be interpreted as the number of cells that are ``skipped'' in the $g_i$'s in the representative.

\begin{theorem}
    The function $\rho$, as defined in \autoref{def:P_rank}, is well-defined and a rank function.
\end{theorem}

\begin{proof}
    First, to prove well-definedness of $\rho$, we need to show  that equivalent representations give the same value of $\rho$. To do this, we first show that horizontally related representations give the same value of $\rho$, and then we show that vertically related representations give the same value of $\rho$.
    
    \begin{enumerate}
        \item Let $\tau$ and $\tau'$ be horizontally related representations. Then $\tau$ can be transformed to $\tau'$ by making substitutions of the form
        \begin{equation*}
            \left( x_i \xrightarrow{g_i} y_i \xleftarrow{\id_{y_i}} y_i \xrightarrow{g_{i+1}} y_{i+1} \right) \leftrightarrow \left( x_i \xrightarrow{g_{i+1} \circ g_i} y_{i+1} \right).
        \end{equation*}
        As $r(g_{i+1} \circ g_i) = r(g_{i+1}) + r(g_i)$ and $r(\id_{y_i}) = 0$, these substitutions preserve the value of $\rho$, so both $\tau$ and $\tau'$ give the same value of $\rho$.
        
        \item Let $\tau$ and $\tau'$ be vertically related representations. Then they form the bottom and top rows of a commutative diagram:

        \begin{center}
        \begin{tikzpicture}
            \diagram{d}{3em}{3em}{
                w & y_0' & x_1' & y_1' & \cdots & x_k' & z \\
                w & y_0 & x_1 & y_1 & \cdots & x_k & z \\
            };
            
            \path[->,font = \scriptsize, midway]
            (d-1-1) edge node[above]{$g_0'$} (d-1-2)
            (d-1-3) edge node[above]{$f_0'$} (d-1-2)
            (d-1-3) edge node[above]{$g_1'$} (d-1-4)
            (d-1-5) edge node[above]{$f_1'$} (d-1-4)
            (d-1-6) edge node[above]{$f_{k-1}'$} (d-1-5)
            (d-1-6) edge node[above]{$g_k'$} (d-1-7);
    
            \path[->,font = \scriptsize, midway]
            (d-2-1) edge node[below]{$g_0$} (d-2-2)
            (d-2-3) edge node[below]{$f_0$} (d-2-2)
            (d-2-3) edge node[below]{$g_1$} (d-2-4)
            (d-2-5) edge node[below]{$f_1$} (d-2-4)
            (d-2-6) edge node[below]{$f_{k-1}$} (d-2-5)
            (d-2-6) edge node[below]{$g_k$} (d-2-7);
    
            \path[->,font = \scriptsize, midway]
            (d-1-2) edge node[left]{$u_0$} (d-2-2)
            (d-1-3) edge node[left]{$v_1$} (d-2-3)
            (d-1-4) edge node[left]{$u_1$} (d-2-4)
            (d-1-6) edge node[left]{$v_k$} (d-2-6);
            \draw[double] (d-1-1) -- (d-2-1);
            \draw[double] (d-1-7) -- (d-2-7);
    
            \draw[double] (d-2-1) -- (d-1-2);
            \draw[double] (d-2-3) -- (d-1-2);
            \draw[double] (d-2-3) -- (d-1-4);
            \draw[double] (d-2-5) -- (d-1-4);
            \draw[double] (d-2-6) -- (d-1-5);
            \draw[double] (d-2-6) -- (d-1-7);
        \end{tikzpicture}
        \end{center}

        We adopt the notation that $v_0 = \id_w$ and $u_k = \id_z$. Then $r(v_0) = r(u_k) = 0$. Now, commutativity of the diagram gives that $r(u_i) + r(f_i') = r(u_i \circ f_i') = r(f_i \circ v_{i+1}) = r(f_i) + r(v_{i+1})$. Similarly, $r(u_i) + r(g_i') = r(g_i) + r(v_i)$. Applying this, we get
        \begin{align*}
            N &+ \sum_{i=0}^{k-1} r(f_i') - \sum_{i=0}^{k} r(g_i') \\
            &= N + \sum_{i=0}^{k-1} \left(r(f_i) + r(u_i) - r(v_{i+1})\right) - \sum_{i=0}^{k} \left(r(g_i) + r(u_i) - r(v_i)\right) \\
            &= N + r(u_k) - r(v_0) + \sum_{i=0}^{k-1} r(f_i) - \sum_{i=0}^{k} r(g_i) \\
            &= N + \sum_{i=0}^{k-1} r(f_i) - \sum_{i=0}^{k} r(g_i). \\
        \end{align*}
        Thus, $\tau$ and $\tau'$ give the same value of $\rho$.
    \end{enumerate}

    Now, we prove that $\rho$ is a rank function. To prove this, we must prove that it satisfies the two conditions in \autoref{def:graded_poset}.
    \begin{enumerate}
        \item We prove that $\rho$ is compatible with the partial order.

        First, let $x$ and $y$ be elements in $P_{w,z}$ different from $\hat 0$ such that $x > y$. Then $x \Rightarrow y$ when considered as elements of $\Hom(w, z)$. We can thus pick representatives $\tau$ and $\tau'$ as in \autoref{thm:flo_po_desc}, and we get that
        \begin{equation*}
            \rho(x) - \rho(y) = \sum_{i=0}^{k} r(g_i') - \sum_{i=0}^{k} r(g_i) = \sum_{i=0}^{k} \left(r(g_i') - r(g_i)\right) > 0,
        \end{equation*}
        where the last inequality follows from the fact that $g_i \Rightarrow g_i'$ and that $g_i \ne g_i'$ for at least one $i$ (otherwise, $x$ would equal $y$).

        Now, we must show that $\rho(x) > \rho(\hat 0)$ for $x \ne 0$. For this, let 
        \begin{equation*}
            \tau = \left(w = x_0 \xrightarrow{g_0} y_0 \xleftarrow{f_0} x_1 \xrightarrow{g_1} \cdots \ \xleftarrow{f_{k-1}} x_k \xrightarrow{g_k} y_k = z\right)
        \end{equation*}
        be a representative of $x$. Observe that $r(g_i) \le \dim x_i - \dim y_i$. Furthermore, $r(f_i) = \dim x_{i+1} - \dim y_i$ (as $\dim x_{i+1} - \dim y_i$ is either 0 or 1, due to the fact that the Morse system $\Sigma$ consists of regular pairs of a discrete Morse function). Hence,
        \begin{equation}\label{eq:rho_inequality}
            \begin{aligned}
                \rho(x) = &N + \sum_{i=0}^{k-1} r(f_i) - \sum_{i=0}^{k} r(g_i) \\
                \ge &N + \sum_{i=0}^{k-1} \left(\dim x_{i+1} - \dim y_i\right) - \sum_{i=0}^{k} \left(\dim x_i - \dim y_i\right) \\
                = &N + \dim y_k - \dim x_0 = N + \dim z - \dim w = 0.
            \end{aligned}
        \end{equation}
        Thus, $\rho(x) \ge 0 > -1 = \rho(\hat 0)$, as desired.
        
        \item We prove that $\rho$ is compatible with the covering relation.

        First, let $x$ and $y$ be elements in $P_{w,z}$ different from $\hat 0$ such that $x$ covers $y$. We again choose representatives $\tau$ and $\tau'$ as in \autoref{thm:flo_po_desc}. It follows that $g_i = g_i'$ for all but one $i$, or else we could choose a representative $\gamma$ so that $x > [\gamma] > y$. Similarly, it follows that for this $i$, $g_i'$ covers $g_i$, or else we could find a $\hat g$ such that $g_i' > \hat g > g_i$, which again would let us choose a representative $\gamma$ so that $x > [\gamma] > y$. As $g_i'$ covers $g_i$, we have that $r(g_i') = r(g_i) + 1$, and thus
        \begin{equation*}
            \rho(x) - \rho(y) = \sum_{i=0}^{k} r(g_i') - \sum_{i=0}^{k} r(g_i) = 1,
        \end{equation*}
        as desired.

        Now, we must show that if $x$ covers $\hat 0$, then $\rho(x) = \rho(\hat 0) + 1 = 0$. An element $x$ covers $\hat 0$ if and only if there are no $z \in \Hom(w, z)$ such that $x \Rightarrow z$.
        Let the following be a representative of such a $x$.
        \begin{equation*}
            \tau = \left(w = x_0 \xrightarrow{g_0} y_0 \xleftarrow{f_0} x_1 \xrightarrow{g_1} \cdots \ \xleftarrow{f_{k-1}} x_k \xrightarrow{g_k} y_k = z\right)
        \end{equation*}
        Then, all $g_i$ must be maximal in $\Hom_{\Ent[X]}(x_i, y_i)$, or else we could replace this $g_i$ to get a representative $\tau'$ with $[\tau] \Rightarrow [\tau']$. It is easy maximal elements in $\Hom_{\Ent[X]}(x_i, y_i)$ are sequences on the form $(x_i = a_0 > a_1 > \dots > a_m = y_i)$ with $m = \dim x_i - \dim y_i$, and hence $r(g_i) = \dim x_i - \dim y_i$. We get that the inequality in \eqref{eq:rho_inequality} is an equality in this case, and $\rho(x) = 0$, as desired.\qedhere
    \end{enumerate}
\end{proof}




%
%

\subsection{Computing the Hom posets}

In this section, we describe an algorithm for computing the Hom posets in the discrete flow category. We still consider the case of a Morse system induced by a discrete Morse function on a regular CW complex, and use the same notation as the previous section. We first state some results that are needed for the algorithm.

First, we prove a useful lemma. Recall that the regular pairs of a discrete Morse function are the pairs of simplices $(x, y)$, where $x \in X^k$ and $y \in X^{k+1}$, such that $x < y$ and $f(x) \ge f(y)$ (the set of these is the induced gradient vector field, $V_f$). The following lemma essentially tells us that we can always assume, without loss of generality, that this inequality is an equality for all regular pairs.

\begin{lemma}\label{lemma:dmf_equality}
    Let $f \colon X \to \R$ be a discrete Morse function. There exists a Forman equivalent discrete Morse function $\tilde f$ such that $\tilde f(x) = \tilde f(y)$ for each regular pair $(x, y)$ of $f$.
\end{lemma}

\begin{proof}
    Define $\tilde f \colon X \to \R$ as follows.
    \begin{equation*}
        \tilde f(x) =
        \begin{cases}
            f(y), &\quad\text{if } (x, y) \in V_f \text{ for some } y, \\ 
            f(x), &\quad \text{ otherwise.} \\
        \end{cases}
    \end{equation*}
    
    It's clear that $\tilde f(x) = \tilde f(y)$ for all regular pairs of $f$, so it remains to show that $\tilde f$ is a discrete Morse function, and that $\tilde f$ and $f$ are Forman equivalent. We prove both of these things by showing that their induced gradient vector fields, $V_{\tilde f}$ and $V_f$ are the same.

    Clearly, $V_f \subseteq V_{\tilde f}$. We need to show that $V_{\tilde f} \subseteq V_f$. First, observe that $\tilde f(x) \le f(x)$ for all $x$.
    
    Now, suppose $(x, y) \in V_{\tilde f}$. Then $x < y$ and $\tilde f(x) \ge \tilde f(y)$. If $\tilde f(y) = f(y)$, then $f(y) = \tilde f(y) \le \tilde f(x) \le f(x)$, so $(x, y) \in V_f$. Suppose $\tilde f(y) \ne f(y)$. Then $(y, z) \in V_f$ for some $z$, and $\tilde f(y) = f(z)$. Let $w \ne y$ be such that $x < w < z$. If $(x, w)$ is a regular pair, then $\tilde f(x) = f(w) < f(z) = \tilde f(y)$, which is a contradiction. If $(x, w)$ is not a regular pair, then $\tilde f(x) \le f(x) < f(w) < f(z) = \tilde f(y)$; also a contradiction. Hence, we cannot have that $f(y) \ne \tilde f(y)$, which completes the proof.
\end{proof}

\begin{theorem}\label{thm:computing_representation}
    A morphism $[\gamma] \in \Hom(w, z)$ is represented by a unique sequence $(w = x_0, x_1, \dots, x_k = z)$ of cells, such that 
    \begin{itemize}
        \item no two elements in the sequence are equal, and
        \item for all $i < k$, either $x_{i+1} < x_i$ or $\{x_i, x_{i+1}\}$ is a regular pair.
    \end{itemize}
\end{theorem}

It is clear that such a sequence always defines a morphism in $\Hom(w, z)$, so this tells us that we can count the morphisms in $\Hom(w, z)$ by counting the simple paths from $w$ to $z$ in a directed graph with face relations and regular pairs as edges.

\begin{proof}
    We get a sequence representation from a representation
    \begin{equation*}
        \gamma = \left(x_0 \xrightarrow{g_0} y_0 \xleftarrow{f_0} x_1 \xrightarrow{g_1} \cdots \ \xleftarrow{f_{k-1}} x_k \xrightarrow{g_k} y_k\right)
    \end{equation*}
    by concatenating all the $g_i$'s. For example,
    \begin{equation*}
        \left(x_0 \xrightarrow{g_0 = (x_0 > x_0^1 > y_0)} y_0 \xleftarrow{f_0 = (x_1 > y_0)} x_1 \xrightarrow{g_1 = (x_1 > x_1^1 > x_1^2 > y_1)} y_1\right)
    \end{equation*}
    becomes
    \begin{equation*}
        (x_0, x_0^1, y_0, x_1, x_1^1, x_1^2, y_1).
    \end{equation*}
    To get a sequence without repeating elements, we choose an appropriate representation. To see that this is possible, first observe that we can eliminate successive repeating elements by removing intermediate identity maps (horizontal relation). Furthermore, by \autoref{lemma:dmf_equality}, given a sequence $(x_0, \dots, x_k)$, there is a discrete Morse function $\tilde f$ such that $\tilde f(x_0) \ge \tilde f(x_1) \ge \dots \ge \tilde f(x_k)$. Thus, if a cell $x'$ repeats in the sequence, the only other cell that can be between the two $x'$ is $y'$ where $\{x', y'\}$ form a regular pair. Now, we can simplify the representation, as per \cite[Remark 2.9]{Nanda}, to turn $(\dots, x_i, x', y', x', x_{i+4}, \dots)$ into $(\dots, x_{i}, x', x_{i+4}, \dots)$.

    It remains to show that for a given morphism, this sequence representation is unique. For this, we show that the sequence representation is decided by the algebraic invariant $I$, as defined in \autoref{def:algebraic_invariant}.

    Let $(w = x_0, x_1, \dots, x_k = z)$ and $(w = y_0, y_1, \dots, y_l = z)$ be sequence representations of $[\tau]$ and $[\sigma]$, such that $I([\tau]) = I([\sigma])$. Now, as $x_0$ appears only once in the first sequence, the coefficient before the atom $(x_0 > x_1)$ (or $(x_1 > x_0)$) in $I([\tau])$ is nonzero. Hence, as $y_0 = x_0$ appears only once in the second sequence, we must have $y_1 = x_1$ for $I([\sigma])$ to equal $I([\tau])$. Similarly, we must have $y_2 = x_2$, and so on, so the sequences must be equal in all places. This completes the proof of uniqueness.
\end{proof}

For the following theorem, recall the definitions of $r$ \eqref{eq:r_def} and $\rho$ (\autoref{def:P_rank}) from the previous section.

\begin{theorem}\label{thm:computing_rank}
    Let $[\gamma] \in \Hom(w, z)$ be a morphism represented by the sequence $(x_0, \dots, x_k)$. Let $I$ be the set of indices $i \in \{0, \dots, k-1\}$ such that $x_i > x_{i+1}$. Then,
    \begin{equation*}
        \rho([\gamma]) = \sum_{i \in I} \left(\dim x_i - \dim x_{i+1} - 1 \right)
    \end{equation*}
\end{theorem}

\begin{proof}
    The morphism $[\gamma]$ is represented by
    \begin{equation*}
        \gamma = \left(x_{j_0} \xrightarrow{g_0} x_{j_1-1} \xleftarrow{f_0} x_{j_1} \xrightarrow{g_1} \cdots \ \xleftarrow{f_{k-1}} x_{j_m} \xrightarrow{g_k} x_{j_{m+1}-1}\right)
    \end{equation*}
    for some sequence $(j_0, \dots, j_{m+1})$ (with $j_0 = 0$ and $j_{m+1}=k+1$). Then,
    \begin{align*}
        \sum_{i \in I} & \left(\dim x_i - \dim x_{i+1} - 1 \right)
        = \sum_{i=0}^{m} \sum_{l=j_i}^{j_{i+1}-2} \left(\dim x_l - \dim x_{l+1} - 1 \right) \\
        &= \sum_{i=0}^{m} \left(\dim x_{j_i} - \dim x_{j_{i+1}-1} - (j_{i+1} - j_i - 1) \right) \\
        &= \dim x_{j_0} - \dim x_{j_{m+1}-1} + \sum_{i=1}^{m} \left(\dim x_{j_i} - \dim x_{j_i-1}\right) - \sum_{i=0}^{m} \left(j_{i+1} - j_i - 1 \right) \\
        &= \dim x_0 - \dim x_k + \sum_{i=0}^{m-1} r(f_i) - \sum_{i=0}^{m} r(g_i) = \rho([\gamma]).\qedhere
    \end{align*}
\end{proof}

\begin{theorem}\label{thm:computing_covering}
    Let $[\gamma]$ and $[\tau]$ be morphisms in $\Hom(w, z)$, such that $\rho([\tau]) = \rho([\gamma]) + 1$. Then, $[\tau]$ covers $[\gamma]$ if and only if the sequence representation of $[\gamma]$ equals the sequence representation of $[\tau]$ with exactly one element added or exactly one element removed.
\end{theorem}

\begin{proof}
    First, suppose $[\tau]$ covers $[\gamma]$. By \autoref{thm:flo_po_desc}, you can get $[\gamma]$ from $[\tau]$ by adding an element $x'$ to the sequence representation of $[\tau]$ and possibly reducing to a non-repeating sequence. There are two cases to consider,
    
    \begin{description}[style=multiline]
    \item[(I)] The element $x'$ was not in the sequence representation of $[\tau]$.
    \item[(II)] The element $x'$ was in the sequence representation of $[\tau]$.
    \end{description}
    
    \begin{description}[labelwidth=1.25cm,leftmargin=!]
        \item[Case I] In this case, the sequence representation of $[\gamma]$ is simply the sequence representation of $[\tau]$ with $x'$ added.
        \sloppy
        \item[Case II] In this case, adding $x'$ to the sequence representation of $[\tau]$ gives something of the form $(x_0, \dots, x_i, x', x_{i+1}, x', x_{i+2}, \dots, x_k)$, which is equivalent to $(x_0, \dots, x_i, x', x_{i+2}, \dots, x_k)$. Hence, the sequence representation of $[\gamma]$ is the sequence representation of $[\tau]$ with $x_{i+1}$ removed.
    \end{description}
    
    Now, for the other direction, there are again two cases to consider.
    \begin{description}[style=multiline]
    \item[(I)] The sequence representation of $[\gamma]$ equals that of $[\tau]$ with one element added.
    \item[(II)] The sequence representation of $[\gamma]$ equals that of $[\tau]$ with one element removed.
    \end{description}

    \begin{description}[labelwidth=1.25cm,leftmargin=!]
        \sloppy
        \item[Case I] Let the sequence representations of $[\tau]$ and $[\gamma]$ be $(x_0, \dots, x_k)$ and $(x_0, \dots, x_i, x', x_{i+1}, \dots, x_k)$, respectively. We use \autoref{thm:computing_rank} and the fact that $\rho([\tau]) = \rho([\gamma]) + 1$ to conclude that $\dim x_i > \dim x' > \dim x_{i+1}$. Hence, $[\gamma]$ equals $[\tau]$ with the element $x'$ added to a right-pointing arrow, so $[\gamma] \Rightarrow [\tau]$.
        
        \sloppy
        \item[Case II] Let the sequence representations of $[\tau]$ and $[\gamma]$ be $(x_0, \dots, x_k)$ and $(x_0, \dots, x_i, x_{i+2}, \dots, x_k)$, respectively. We use \autoref{thm:computing_rank} and the fact that $\rho([\tau]) = \rho([\gamma]) + 1$ to conclude that either $\dim x_{i+1} > \dim x_i$ or $\dim x_{i+2} > \dim x_{i+1}$. In the first case, $[x_i > x_{i+2}] = [x_i < x_{i+1} > x_i > x_{i+2}]$. In the second case, $[x_i > x_{i+2}] = [x_i > x_{i+2} > x_{i+1} < x_{i+2}]$. In any case, $[\gamma]$ can be constructed by adding a single element (either $x_i$ or $x_{i+2}$) to a right-pointing arrow of $[\tau]$, so $[\gamma] \Rightarrow [\tau]$.
    \end{description}

    When $[\gamma] \Rightarrow [\tau]$, $[\gamma]$ covers $[\tau]$, as their rank differ by one, and this concludes the proof.
\end{proof}

\renewcommand{\algorithmicrequire}{\textbf{Input:}}
\renewcommand{\algorithmicensure}{\textbf{Output:}}
\begin{algorithm}
\caption{Compute Hom poset}\label{algorithm_1}
\begin{algorithmic}[1]
\Require{A regular CW complex $X$, a gradient vector field $V_f$, a source $w$, and a target $z$.}
\Ensure The Hom poset $\Hom(w, z)$ of the discrete flow category.
\State Construct a directed graph $G$ with cells as vertices and an edge $u \to v$ if $v < u$ or if $\{u, v\}$ is a regular pair
\State In $G$, find all paths with non-repeating vertices from $w$ to $z$
\State For all paths, compute the rank with \autoref{thm:computing_rank}.
\For{$i = 1..(\dim w - \dim z)$}
  \State Compute the covering relations between paths of rank $i-1$ and paths of rank $i$, using \autoref{thm:computing_covering}.
\EndFor
\end{algorithmic}
\end{algorithm}

Note that the only data needed for the CW complex $X$ is the graph of its covering relations. Only the covering relations of $\Hom(w, z)$ are output with this algorithm, but the full set of partial order relations is easily computed from this.

Regarding complexity of this algorithm, the main computational cost is comparing the covering relations. Comparing two paths to determine if one covers the other takes $\bigo(l)$ time, where $l$ is the length of the longest path. For each path $\gamma$, the algorithm makes one comparison to each path of rank $\rho(\gamma) + 1$. Hence, the algorithm makes $\bigo(n \cdot W)$ comparisons, where $n$ is the number of morphisms in $\Hom(w,z)$, and $W$ is the maximum number of morphisms of a given rank (the maximal ``width'' of a rank layer in the Hasse diagram). In total, the complexity of the algorithm is $\bigo(l \cdot n \cdot W)$ (it's easy to verify that steps 1.-3. have smaller complexity than this). As $W$ is clearly smaller than $n$, we could write this more simply as $\bigo(l \cdot n^2)$. For comparison, listing the sequence representations of all morphisms in $\Hom(w, z)$ takes $\Theta(l \cdot n)$ time (and this is thus a lower bound of how fast an algorithm to compute $\Hom(w,z)$ can be).

\subsection{The Hom posets are CW posets}

In this section, we prove \autoref{theoremA}: that in the case where $X$ is a simplicial complex, then for any $w, z \in \Flo_{\Sigma}[X]$ with $\Hom(w, z)$ nonempty, the poset $P_{w, z}$, as defined in \eqref{eq:Pwz_def}, is a CW poset. (Note that when $\Hom(w, z)$ is empty, then $P_{w,z}$ is just the poset with one element.)

\begin{lemma}\label{lemma:poset_join}
    Let $P$ and $Q$ be posets with greatest elements $\hat 1_P$ and $\hat 1_Q$, respectively. Suppose that $|P \setminus \{\hat 1_P\}| \cong S^m$ and $|Q \setminus \{\hat 1_Q\}| \cong S^n$. Then the realization of the poset $(P \times Q) \setminus \{(\hat 1_P, \hat 1_Q)\}$ is homeomorphic to $S^{m+n+1}$.
\end{lemma}

\begin{proof}
    From \cite[Proposition 1.9]{Quillen}, the realization of $(P \times Q) \setminus \{(\hat 1_P, \hat 1_Q)\}$ is the join of $S^m$ and $S^n$, which is $S^{m+n+1}$.
\end{proof}

\begin{definition}
 \textcolor{black}{
    Let $X$ be a simplicial complex, and let $\sigma \in X$. We define the \emph{link} of $\sigma$ as the set
    \begin{equation*}
        \cof \sigma = \{\tau \setminus \sigma : \sigma \subsetneq \tau \}
    \end{equation*}
}
\end{definition}

 \textcolor{black}{
In other words, the link of $\sigma$ contains all the proper cofaces of $\sigma$, with $\sigma$ subtracted as a set from the complexes. It is a well-known fact that the link of simplex is itself a simplicial complex.
}

\begin{lemma}\label{lemma:ent_hom_realization}
    Let $X$ be a simplicial complex, and let $\sigma \in X^n$ and $\tau \in X^m$ with $\tau \subseteq \sigma$. Then the realization of $\Hom_{\Ent[X]}(\sigma, \tau) \setminus \{(\sigma > \tau)\}$ is homeomorphic to $S^{n-m-2}$.
\end{lemma}

\begin{proof}
    The poset $\Hom_{\Ent[X]}(\sigma, \tau) \setminus \{(\sigma > \tau)\}$ contains all sequences $(\sigma > \upsilon_1 > \dots > \upsilon_k > \tau)$, with $k \ge 1$, and the partial order is given by inclusion. This poset is isomorphic to the poset of descending sequences $(\upsilon_1 > \dots > \upsilon_k)$ of simplices in $\partial (\sigma \setminus \tau) \subseteq \cof \tau$, i.e., the boundary of $\sigma \setminus \tau$ in the  \textcolor{black}{link} of $\tau$ (where the ordering is still given by inclusion). This is again isomorphic to the barycentric subdivision $T(\partial (\sigma \setminus \tau))$.

    Now,
    \begin{equation*}
        | T(\partial (\sigma \setminus \tau)) | \cong | \partial (\sigma \setminus \tau) | \cong | \partial \Delta^{n-m-1} | \cong |S^{n-m-2}|,
    \end{equation*}
    as desired (here we use that $\sigma \setminus \tau$ is a set of cardinality $(n+1)-(m+1) = n-m$, and hence an $(n-m-1)$--simplex).
\end{proof}

In the proof of the following lemma, we use the fact that for a poset $P$, $|P^{\op}| \cong |P|$, which is a special case of the fact that $| \N (C^{\op}) | \cong | \N(C) |$ for a general category $C$.

\begin{lemma}\label{lemma:ent_hom_interval}
    Let $X$ be a simplicial complex. Let $P = \Hom_{\Ent[X]}(w, z)^{\op} \cup \{\hat 0\}$, where $\hat 0$ is a least element. Then, for any $g \in P \setminus \{\hat 0\}$, the realization of the open interval $(\hat 0, g)$ is homeomorphic to a sphere. 
\end{lemma}

\begin{proof}
    Let $g = (x_0, \dots, x_k)$. Then,
    \begin{equation*}
        (\hat 0, g] \quad \cong \quad \Hom(x_0, x_1)^{\op} \times \Hom(x_1, x_2)^{\op} \times \dots \times \Hom(x_{k-1}, x_k)^{\op},
    \end{equation*}
    where the $\Hom(x_i, x_{i+1})$ are Hom posets in $\Ent[X]$. Now, observe that the atom $(x_i > x_{i+1})$ is the greatest element in $\Hom(x_i, x_{i+1})^{\op}$. Furthermore, by \autoref{lemma:ent_hom_realization}, the realization of $\Hom(x_i, x_{i+1})^{\op} \setminus \{(x_i > x_{i+1})\}$ is homeomorphic to a sphere. Thus, by applying \autoref{lemma:poset_join} to the equation above, we get that the realization $|(\hat 0, g)| = |(\hat 0, g] \setminus \{g\}|$ is homeomorphic to a sphere.
\end{proof}

\begin{theorem}[Theorem A]\label{thm:hom_op_is_cw}
    In the case that $X$ is a simplicial complex, the poset $P_{w, z}$, as defined in \eqref{eq:Pwz_def}, is a CW poset. 
\end{theorem}

\begin{proof}
    We use \autoref{def:cw_poset}. Conditions (1.) and (2.) are clearly fulfilled, so we need only to show that for all $x \in P_{w, z} \setminus \{\hat 0\}$, the realization of the open interval $(\hat 0, x)$ is homeomorphic to a sphere.

    Let $x$ be represented by 
    \begin{equation*}
        \tau = \left(w = x_0 \xrightarrow{g_0} y_0 \xleftarrow{f_0} x_1 \xrightarrow{g_1} \cdots \ \xleftarrow{f_{k-1}} x_k \xrightarrow{g_k} y_k = z\right).
    \end{equation*}
    Applying \autoref{thm:flo_po_desc}, we see that
    \begin{equation*}
        (\hat 0, x] \quad \cong \quad (\hat 0, g_0] \times (\hat 0, g_1] \times \dots \times (\hat 0, g_k].
    \end{equation*}
    By \autoref{lemma:ent_hom_interval}, the realization of $(0, g_i)$ is a sphere. Now, as in the proof for \autoref{lemma:ent_hom_interval}, we apply \autoref{lemma:poset_join} to the equation above and get that $|(\hat 0, x)|$ is homeomorphic to a sphere.
\end{proof}

In conclusion, what this result tells us is that in the discrete flow category, the opposite poset of a Hom set is the face poset of a regular CW complex. The nerve of the face poset of a regular CW complex is its barycentric subdivision, and hence has a geometric realization which is homeomorphic to the CW complex. Furthermore, the nerve of the opposite of a poset is isomorphic to the nerve of the poset. Thus, we get a simpler way to describe the geometric realization of the nerve of a Hom set in the discrete flow category: instead of taking the nerve and realizing, we can construct the CW complex of which it is a face poset, which results in a description with fewer cells.
To construct this CW complex, we need only compute the rank values and the covering relations, which we can do with \autoref{algorithm_1}.

\begin{example}\label{ex:D3_flo}

We illustrate the observations in the last paragraph with an example.

Consider a regular CW decomposition of the filled sphere $D^3$, as illustrated in \autoref{fig:filled_sphere}. We define the gradient vector field $V_f = \{(w > y), (b > z)\}$, and let $\Sigma$ be its corresponding Morse system. We compute $\Hom(f, x)$ in the discrete flow category.

\begin{figure}[htbp]
    \centering
    \begin{tikzpicture}[scale=.5]
        \filldraw[lightgray] (0,0) circle (5);
        \draw[] (0,0) circle (5);
        \draw[thick] (0,0) ellipse (5 and 2);
        \filldraw (5,0) circle (5 pt) node[anchor=west]{$y$};
        \filldraw (-5,0) circle (5 pt) node[anchor=east]{$x$};
        \draw (0, 2.5) node{$w$};
        \draw (0, -1.5) node{$z$};
        \draw (0, 5.5) node{$t$};
        \draw (0, -5.5) node{$b$};
        \draw (6.5, 5) node{$f$};
        \draw [->, dashed] (6, 5) .. controls (4, 5) and (2, 3) .. (1, 0);
    \end{tikzpicture}
    \caption{A CW decomposition of $D^3$. The 0--cells are $x$ and $y$, the 1--cells are $w$ and $z$, the 2--cells are $t$ and $b$ and the single 3--cell is $f$.}
    \label{fig:filled_sphere}
\end{figure}
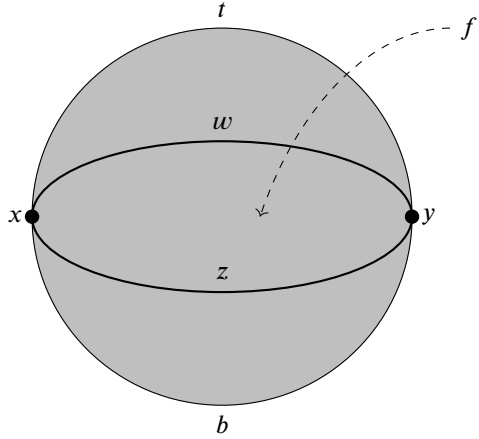

As we have shown that $\Hom(f, x)$ is a graded poset, we can represent it by a Hasse diagram. The highest rank is $\dim f - \dim x - 1 = 2$. All rank 2 elements are gradient paths on the form 
\begin{equation*}
    f = x_0 \xrightarrow{g_0} y_0 \xleftarrow{f_0} x_1 \xrightarrow{g_1} \cdots \ \xleftarrow{f_{k-1}} x_k \xrightarrow{g_k} y_k = x,
\end{equation*}
where all $g_i$ are atoms. There are four of these:
\begin{itemize}
    \item $f > x$
    \item $f > z < b > x$
    \item $f > y < w > x$
    \item $f > z < b > y < w > x$
\end{itemize}
All elements of lower rank can be constructed from these by inserting intermediate cells, and all Hasse diagram edges correspond to inserting an intermediate cell in a sequence. For example, inserting $b$ between $f$ and $x$ in $f > x$ gives $f > b > x$, and inserting $b$ between $f$ and $z$ in $f > z < b > x$ gives $f > b > z < b > x = f > b > x$. We now draw the Hasse diagram, which is illustrated in \autoref{fig:D3_hom_poset}.

\begin{figure}[htbp]
    \centering
    \includegraphics[width=1\textwidth]{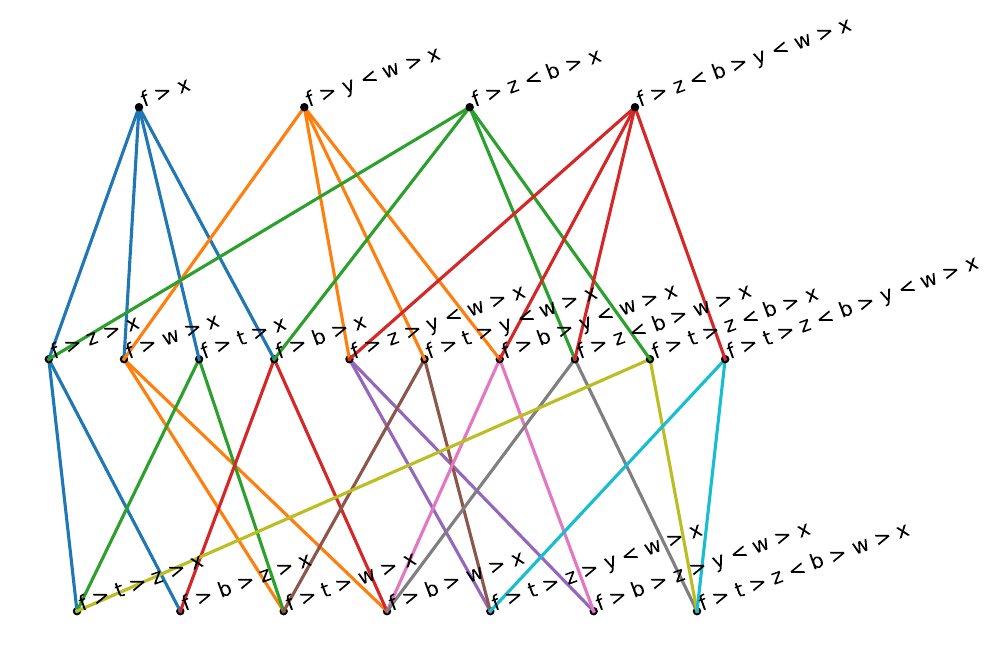}
    \caption{The Hasse diagram for the Hom poset $\Hom(f, x)$ in the discrete flow category.}
    \label{fig:D3_hom_poset}
\end{figure}

The corresponding regular CW complex is illustrated in \autoref{fig:D3_hom_CW_cplx}, and it's easy to see that this is a contractible space.

\begin{figure}[htbp]
    \centering
    \begin{tikzpicture}
        \filldraw[lightgray] (-4, 0) -- (0, -3) -- (4, 0) -- (0, 3) -- cycle;

        \draw (-4, 0) -- (4, 0);
        \draw (0, -3) -- (0, 3);
        \draw (0, -3) -- (4, 0);
        \draw (0, -3) -- (-4, 0);
        \draw (0, 3) -- (4, 0);
        \draw (0, 3) -- (-4, 0);
        
        \filldraw (0, 0) circle (2 pt);
        \filldraw (2, 0) circle (2 pt);
        \filldraw (4, 0) circle (2 pt);
        \filldraw (-2, 0) circle (2 pt);
        \filldraw (-4, 0) circle (2 pt);
        \filldraw (0, 3) circle (2 pt);
        \filldraw (0, -3) circle (2 pt);
    \end{tikzpicture}
    \caption{The face poset of this regular CW complex is $\Hom(f, x)$.}
    \label{fig:D3_hom_CW_cplx}
\end{figure}
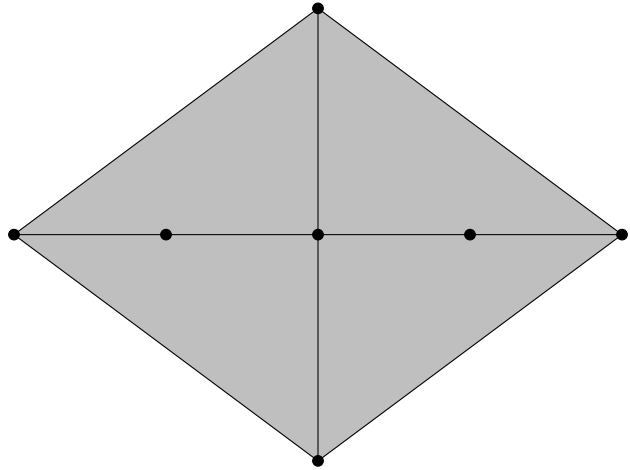

\end{example}

\begin{example}\label{ex:theoremAcounter}
We illustrate that \autoref{thm:hom_op_is_cw} may fail to hold when $X$ is not a simplicial complex. In fact, \autoref{lemma:ent_hom_interval}, which may be viewed as a special case of \autoref{thm:hom_op_is_cw} where the Morse system is empty (so that $\Flo_{\Sigma}[X] = \Ent[X]$), fails to hold in our example.

First, we describe how the suspension of a simplicial complex is also a simplicial complex. Let $X$ be a simplicial complex. Its cone $CX$ can be viewed as a simplicial complex as follows:
\begin{equation*}
    CX = X \cup \{\{0\}\} \cup \{\sigma \cup \{0\} : \sigma \in X\}.
\end{equation*}
The suspension $\Sigma X$, which is just the union of two cones $CX$ along $X$, can then be viewed as the simplicial complex:
\begin{equation}\label{eq:suspension_cplx}
    \Sigma X = X \cup \{\{0\}\} \cup \{\sigma \cup \{0\} : \sigma \in X\}
    \cup \{\{1\}\} \cup \{\sigma \cup \{1\} : \sigma \in X\}.
\end{equation}

Our counterexample uses the following fact: there exists a simplicial complex $Y$ such that the suspension $\Sigma Y$ is homeomorphic to a sphere, but $Y$ is not homeomorphic to $S^k$ for any $k$.
For this, we will consider the Poincaré sphere, which is a simplicial complex and a homology 3--sphere, but not homeomorphic to a sphere. Its suspension is not homeomorphic to a sphere, but its double suspension is (for more details, see \cite[Example 3.2.11]{Thurston}). Hence, letting $Y$ be the suspension of the Poincaré sphere, we get our desired properties.

Now, let $Y$ be the suspension of the Poincaré sphere. As $\Sigma Y$ is a simplicial complex and homeomorphic to $S^5$, it is also a regular CW complex of dimension 4. Hence we can construct a new regular CW complex $Z$ by attaching a 5--cell $e$ along $\Sigma Y$. The CW complex $Z$ is illustrated in \autoref{fig:poincare_counterexample}.

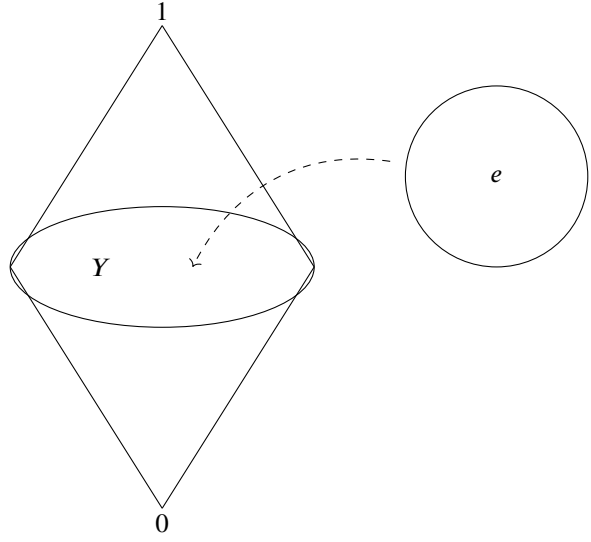
\begin{figure}[htbp]
    \centering
    \begin{tikzpicture}[scale=.4]
        \draw[] (0,0) ellipse (5 and 2);
        \draw (-5,0) -- (0,-8);
        \draw (5,0) -- (0,-8);
        \draw (-5,0) -- (0,8);
        \draw (5,0) -- (0,8);
        \draw (0, 8.5) node{$1$};
        \draw (0, -8.5) node{$0$};
        \draw (11, 3) circle (3);
        \draw (11, 3) node{$e$};
        \draw [->, dashed] (7.5, 3.5) .. controls (4, 4) and (2, 2) .. (1, 0);
        \draw (-2,0) node{$Y$};
    \end{tikzpicture}
    \caption{An illustration of the CW complex $Z$.}
    \label{fig:poincare_counterexample}
\end{figure}

Now, consider the Hom poset
\begin{equation}\label{eq:hom_counterexample}
    \Hom_{\Ent[Z]}(e, 0) \setminus \{(e > 0)\},
\end{equation}
where 0 is a 0--simplex in $\Sigma Y$ as per the notation in \eqref{eq:suspension_cplx}. This poset consists of all descending chains $(e > \tau_1 > \dots > \tau_k > 0)$, with $k \ge 1$, ordered by inclusion. The set of $\tau_i$ satisfying $e > \tau_i > 0$ is precisely $\{\{0\} \cup \sigma : \sigma \in Y\}$. Hence, the poset in \eqref{eq:hom_counterexample} is isomorphic to the barycentric subdivision of $Y$. Thus,
\begin{equation*}
    | \Hom_{\Ent[Z]}(e, 0) \setminus \{(e > 0)\} | \cong |T(Y)| \cong |Y|,
\end{equation*}
which is not homeomorphic to a sphere. Thus, the interval $(\hat 0, (e>0)) \in P_{e,0}$ (which is just the opposite of the poset in \eqref{eq:hom_counterexample}) is not a sphere, and hence $P_{e, 0}$ is not a sphere.

\end{example}

\section{Simplicial collapse on simplicial sets}

In this section we define \emph{regular} simplicial sets, and generalize the concept of \emph{simplicial collapse} to these simplicial sets. We further prove how the nerves of certain categories are regular, and apply simplicial collapse to a subclass of these categories.

\subsection{Collapses on simplicial complexes}

In \autoref{sec:simplicial_collapse}, we defined \emph{free faces} and \emph{collapses}.

We now provide a more general definition of free faces, that allows for codimension greater than 1. 

\begin{definition}\label{def:free_pair}
    Let $X$ be a simplicial complex and let $\tau$ and $\sigma$ be simplices in $X$ such that
    \begin{enumerate}
        \item $\tau \subsetneq \sigma$, and
        \item all cofaces of $\tau$ are faces of $\sigma$.
    \end{enumerate}
    Then $\tau$ is called a \emph{free face} and $\{\tau, \sigma\}$ is a \emph{free pair}.
\end{definition}

Recall that in our old definition of free pairs, removing a free pair of $X$ from a simplicial complex constitutes an elementary collapse. Recall that we write $X \searrow Y$ if $Y$ can be reached through a series of elementary collapses on $X$. The following theorem justifies our generalization of the definition of free pairs.
\footnote{ \textcolor{black}{Note that these collapses fit into the framework of \emph{generalized Morse theory} (see, e.g., \cite{Bauer_2016}).}}

\begin{theorem}{{\cite[Proposition 9.18]{OrganizedCollapse}}}\label{thm:compound_collapse}
    Let $X$ be a simplicial complex and let $\{\tau, \sigma\}$ be a free pair according to \autoref{def:free_pair}. Let $Y \setminus [\tau, \sigma]$ be the simplicial complex where all simplices $\gamma$ satisfying $\tau \subseteq \gamma \subseteq \sigma$ have been removed from $X$. Then $Y$ is a simplicial complex, and $X \searrow Y$.
\end{theorem}

Note that in the previous theorem, we might as well have written $\tau \subseteq \gamma$ instead of $\tau \subseteq \gamma \subseteq \sigma$, as all cofaces of $\tau$ are faces of $\sigma$.

An example of a simplicial collapse of the kind described in \autoref{thm:compound_collapse} is given in \autoref{fig:compound_collapse}.
\begin{figure}[htbp]
    \centering
    \begin{tikzpicture}[scale=.7]
        \filldraw[lightgray] (0,0) -- (4,0) -- (2, 3.4) -- cycle;
        
        \draw (0, 0) -- (4, 0);
        \draw (0, 0) -- (2, 3.4);
        \draw (4, 0) -- (2, 3.4);
        
        \filldraw (0, 0) circle (2 pt) node[anchor=north]{$a$};
        \filldraw (4, 0) circle (2 pt) node[anchor=north]{$b$};
        \filldraw (2, 3.4) circle (2 pt) node[anchor=east]{$c$};
        
        \draw[->, very thick] (0.75, 1.3) -- (1.75, 1.9);
        \draw[->, very thick] (0, 0) -- (2, 1.2);
        \draw[->, very thick] (1.5, 0) -- (2.5, 0.6);
        
        \draw[very thick, ->, double] (5.5, 1.7) -- (6.5, 1.7);
        
        \draw (11, 0) -- (9, 3.4);
        
        \filldraw (11, 0) circle (2 pt) node[anchor=east]{$b$};
        \filldraw (9, 3.4) circle (2 pt) node[anchor=east]{$c$};
    \end{tikzpicture}
    \caption{A simplicial collapse given by the free pair $\{\{a\}, \{a,b,c\}\}$.}
    \label{fig:compound_collapse}
\end{figure}
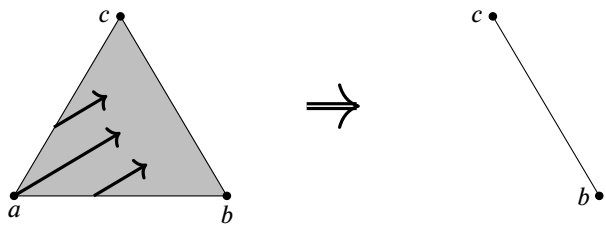

\subsection{Collapses on simplicial sets}

We generalize free pairs and collapses to simplicial sets. We will restrict ourselves to a certain class of simplicial sets, which we will name \emph{regular} simplicial sets.

\begin{definition}
    Let $X$ be a simplicial set. Then $X$ is called \emph{regular} if for all $n \in \mathbb{N}$ (including 0) and all nondegenerate $x \in X_n$, the map
    \begin{align*}
        \Hom_{\Delta}([0], [n]) &\to X_0 \\
        \theta &\mapsto X(\theta)(x)
    \end{align*}
    is injective.
\end{definition}

This definition says that for a nondegenerate $n$--simplex, all of the $(n+1)$ 0--dimensional faces are different. This also implies that for each nondegenerate $n$--simplex $x$ and all $m < n$, all of the $m$--dimensional faces of $x$ are different\footnote{More precisely, the map $\Hom_{\Delta}([m], [n]) \to X_m$ that maps $\theta$ to $X(\theta)(x)$ is injective.}, because no two $m$--dimensional faces have the same 0--dimensional faces. In other words, for $x \in X_n$ nondegenerate, the simplicial set consisting of $x$, all the faces of $x$, and all their degeneracies, is isomorphic to $\Delta^n$. This last part is the motivation behind the definition. Just as for simplicial complexes, the sub-simplicial set generated by a nondegenerate $n$--simplex $x$ looks like $\Delta^n$, which will allow us to define free pairs and collapses, just as for simplicial complexes.

\begin{example}
An example of a regular simplicial set is given in \autoref{fig:regular_sset}. Observe how there are two 1--simplices with $b$ and $c$ as faces, something that is not possible for simplicial complexes. However, just as for simplicial complexes, all 1--simplices lie between two different 0--simplices, and the 2--simplex $A$ looks like a 2--simplex in a simplicial complex. To make the last point more precise, the simplicial set consisting of $A$ and all its faces and degeneracies, is isomorphic to $\Delta^2$.

\begin{figure}[htbp]
    \centering
    \begin{tikzpicture}[scale=.7]
        \filldraw[lightgray] (0,0) -- (4, 0) .. controls (3.3, 0.4) and (2, 2.4) .. (2, 3.4) -- cycle;
        \draw (1.5, 1) node {$A$};
        
        \draw (0, 0) -- (4, 0) node[midway, anchor=north]{$\alpha$};
        \draw (0, 0) -- (2, 3.4) node[midway, anchor=east]{$\delta$};
        \draw (4, 0) .. controls (3.3, 0.4) and (2, 2.4) .. (2, 3.4) node[midway, anchor=east]{$\beta$};
        \draw (4, 0) .. controls (4, 1) and (2.7, 2.8) .. (2, 3.4) node[midway, anchor=west]{$\gamma$};
        
        \filldraw (0, 0) circle (2 pt) node[anchor=north]{$a$};
        \filldraw (4, 0) circle (2 pt) node[anchor=north]{$b$};
        \filldraw (2, 3.4) circle (2 pt) node[anchor=east]{$c$};
    \end{tikzpicture}
    \caption{A regular simplicial set (degeneracies omitted).}
    \label{fig:regular_sset}
\end{figure}
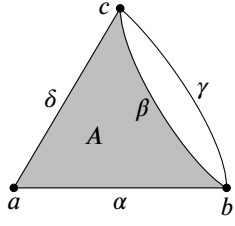
\end{example}

\begin{lemma}\label{lemma:regular_sset_realization}
    The realization of a regular simplicial set is a regular CW complex.
\end{lemma}

\begin{proof}
    Let $X$ be a regular simplicial set. Recall that a CW complex is regular if all attaching maps are homeomorphisms onto their image. From \cite[Theorem 14.1]{May}, we know that $|X|$ is a CW complex with one $n$--cell for each nondegenerate simplex. Furthermore, for an $n$--cell corresponding to a nondegenerate $n$--simplex $x$, the attaching map is given by sending $|\partial \Delta^n|$ to $|\partial x|$ (in the canonical way). Now, from the observations in the above paragraph, we know that this map is a homeomorphism, so we are finished.
\end{proof}

We introduce some new notation. We say that $\tau \in X_m$ is a \emph{face} of $\sigma \in X_n$ if $\tau = X(\theta)(\sigma)$ for some $\theta \in \Hom_{\Delta}([m], [n])$. Note that in this definition, degeneracies are also considered faces, but the nondegenerate faces are just as for simplicial complexes. We say that $\sigma$ is a coface of $\tau$ when $\tau$ is a face of $\sigma$. We write $\tau \subseteq \sigma$ when $\tau$ is a face of $\sigma$, and we write $\tau \subsetneq \sigma$ when $\tau$ is a \emph{proper} face of $\sigma$, i.e., when $\tau \subseteq \sigma$ and $\tau \neq \sigma$. It's clear that this relation is transitive, as if $\tau \subseteq \tau'$ and $\tau' \subseteq \tau''$, then $\tau = X(\theta)(\tau')$ and $\tau' = X(\theta')(\tau'')$ for some $\theta, \theta'$, so $\tau = X(\theta \circ \theta')(\tau'')$. It's not, however, antisymmetric, as all simplices are faces of all their degeneracies.

\begin{example}
In the simplicial set in \autoref{fig:regular_sset}, the faces of $A$ are
\begin{itemize}
    \item $A$ itself,
    \item the degeneracies of $A$,
    \item the nondegenerate simplices $a$, $b$, $c$, $\alpha$, $\beta$ and $\delta$,
    \item the degeneracies of the above mentioned simplices.
\end{itemize}
The simplex $\gamma$, on the other hand, is \emph{not} a face of $A$.
\end{example}

\begin{definition}\label{def:free_pair_sset}
    Let $X$ be a regular simplicial set, and let $\tau$ and $\sigma$ be nondegenerate simplices such that
    \begin{enumerate}
        \item $\tau \subsetneq \sigma$, and
        \item all cofaces of $\tau$ are faces of $\sigma$.
    \end{enumerate}
    Then $\tau$ is called a \emph{free face} and $\{\tau, \sigma\}$ is a \emph{free pair}.
\end{definition}

Note that the fact that $\tau$ and $\sigma$ are nondegenerate and that $\tau \subsetneq \sigma$ implies that $\dim \tau < \dim \sigma$.

\begin{theorem}
    Let $X$ be a simplicial set and let $\{\tau, \sigma\}$ be a free pair according to \autoref{def:free_pair}. Let $Y \setminus [\tau, \sigma]$ be the simplicial set where all simplices $\gamma$ satisfying $\tau \subseteq \gamma \subseteq \sigma$ have been removed from $X$. Then $Y$ is a simplicial set, and the realization $|Y|$ is a deformation retract of $|X|$.
\end{theorem}

\begin{proof}
    We first prove that $Y$ is a simplicial set. We must show that for all $x \in Y$, all faces (including degeneracies) of $x$ are contained in $Y$. Let $x$ be an $n$--simplex in $Y$, and suppose $y$ is a face of $x$ not in $Y$. Then $y$ is a coface of $\tau$. But then $x$ is also a coface of $\tau$ (and hence also a face of $\sigma)$, so $x \notin Y$, which is a contradiction. This proves that $Y$ is a well-defined simplicial set.

    We now show that $|Y|$ is a deformation retract of $|X|$. From \autoref{lemma:regular_sset_realization}, we know that $|X|$ is a regular CW complex, and so is $|Y|$. Let $e_{\tau}$ and $e_{\sigma}$ denote the cells in $|X|$ corresponding to $\tau$ and $\sigma$, respectively. Then $\overline{e_{\sigma}}$ (i.e., the closure of $e_{\sigma}$) is homeomorphic to $|\Delta^n|$ (for some $n$), and under this homeomorphism, $e_{\tau}$ maps to some face of $|\Delta^n|$. Thus, $\overline{e_{\sigma}}$ deformation retracts to $\left( \overline{e_{\sigma}} \setminus \{e' : e_{\tau} \subseteq e'\} \right)$ (i.e., $\overline{e_{\sigma}}$ with all the cofaces of $e_{\tau}$ removed), just as for simplicial complexes. This deformation retract then extends to a deformation retract from $|X|$ to $|Y|$ by setting it to be identity everywhere else.
\end{proof}

When $X$ deformation retracts to $Y$ in this way, we call it a \emph{collapse}, and say that $X$ \emph{collapses} to $Y$.

\begin{example}
In \autoref{fig:sset_collapse}, $\{a, A\}$ is a free pair in the regular simplicial set on the left hand side. The right hand side shows the simplicial set after the corresponding collapse. Observe how the collapse gives a deformation retract on the closure of the 2--cell corresponding to $A$, which extends to a deformation retract on the entire realization by defining it to be identity everywhere else.
    
\begin{figure}[htbp]
    \centering
    \begin{tikzpicture}[scale=.7]
        \filldraw[lightgray] (0,0) -- (4, 0) .. controls (3.3, 0.4) and (2, 2.4) .. (2, 3.4) -- cycle;
        \draw (1.5, .7) node {$A$};
        
        \draw (0, 0) -- (4, 0) node[midway, anchor=north]{$\alpha$};
        \draw (0, 0) -- (2, 3.4) node[midway, anchor=east]{$\delta$};
        \draw (4, 0) .. controls (3.3, 0.4) and (2, 2.4) .. (2, 3.4) node[midway, anchor=east]{$\beta$};
        \draw (4, 0) .. controls (4, 1) and (2.7, 2.8) .. (2, 3.4) node[midway, anchor=west]{$\gamma$};
        
        \filldraw (0, 0) circle (2 pt) node[anchor=north]{$a$};
        \filldraw (4, 0) circle (2 pt) node[anchor=north]{$b$};
        \filldraw (2, 3.4) circle (2 pt) node[anchor=east]{$c$};
        
        \draw[->, very thick] (0.75, 1.3) -- (1.75, 1.9);
        \draw[->, very thick] (0, 0) -- (2, 1.2);
        \draw[->, very thick] (1.5, 0) -- (2.5, 0.6);
        
        \draw[very thick, ->, double] (5.5, 1.7) -- (6.5, 1.7);
        
        \draw (11, 0) .. controls (10.3, 0.4) and (9, 2.4) .. (9, 3.4) node[midway, anchor=east]{$\beta$};
        \draw (11, 0) .. controls (11, 1) and (9.7, 2.8) .. (9, 3.4) node[midway, anchor=west]{$\gamma$};
        
        \filldraw (11, 0) circle (2 pt) node[anchor=east]{$b$};
        \filldraw (9, 3.4) circle (2 pt) node[anchor=east]{$c$};
    \end{tikzpicture}
    \caption{A collapse of a regular simplicial set.}
    \label{fig:sset_collapse}
\end{figure}
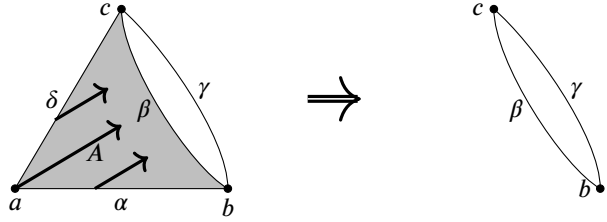
\end{example}

\subsection{Unique factorization categories}

We now define a certain class of categories, which we will call \emph{unique factorization categories}. In these categories, all morphism can be written as a composition of one or more \emph{indecomposable} morphisms in a unique way. We shall see that these categories have nerves that are regular simplicial sets, and with many free faces.
\footnote{ \textcolor{black}{Note that for the categories we consider in this section, we could have chosen to work with a different notion of nerve, that gives a \emph{semi-simplicial set} (see, e.g., \cite[Chapter 10.2]{KozlovCombinatorial}). However, to not change paradigms within the article, we choose to stick to simplicial sets.}}

\begin{definition}
    Let $C$ be a category and $f \colon A \to B$ a morphism in $C$. Then $f$ is called \emph{indecomposable} if it cannot be written as a composite of non-identity morphisms, i.e., if there exist no morphisms $h \neq \id_B$ and $g \ne \id_A$ such that $f = h \circ g$.
\end{definition}

We say that a category is \emph{finite} if both its set of objects and set of morphisms is finite.

\begin{definition}\label{def:finite_directed_cat}
    A \emph{finite directed category} is a finite category such that:
    \begin{itemize}
        \item For all objects $a$, $\Hom(a,a)$ contains only the identity $\id_a$.
        \item For all objects $a \neq b$, if $\Hom(a, b)$ is nonempty, then $\Hom(b,a) = \emptyset$.
    \end{itemize}
\end{definition}

Any finite poset (viewed as a category) is an example of a finite directed category.

It's clear that for a finite directed category, each morphism $f$ can be written as a composition $f = f_k \circ \cdots \circ f_1$ of indecomposable morphisms. However, its possible that there are several different ways to do this. For example, consider the face poset of the simplicial complex $\Delta^2$. Here, the morphism $(\{0\} \le \{0,1,2\})$ has two decompositions: $(\{0,1\} \le \{0,1,2\}) \circ (\{0\} \le \{0,1\})$ and $(\{0,2\} \le \{0,1,2\}) \circ (\{0\} \le \{0,2\})$

\begin{proposition}
    Let $C$ be a finite directed category. Then the nerve of $C$ is a regular simplicial set.
\end{proposition}

\begin{proof}
    A nondegenerate $n$--simplex in $\N(C)$ is a set of composable non-identity morphisms $\{f_i \colon x_i \to x_{i+1} : 0 \le i < n\}$. Suppose $x_i = x_j$ for some $i < j$. If $j = i+1$, then $f_i$ is a non-identity morphism in $\Hom(x_i, x_i)$, which contradicts \autoref{def:finite_directed_cat}. If $j > i+1$, then both $\Hom(x_i, x_{i+1})$ and $\Hom(x_{i+1}, x_i)$ are nonempty, which also contradicts \autoref{def:finite_directed_cat}.

    Hence, any nondegenerate $n$--simplex in $\N(C)$ has $n+1$ different 0--dimensional faces, and so $\N(C)$ is regular.
\end{proof}

\begin{definition}
    A \emph{unique factorization category} is a finite directed category $C$ such that all non-identity morphisms $f$ in $C$ can be written as a composition of non-identity indecomposable morphisms in a unique way.
\end{definition}

\begin{example}\label{ex:unique_factorization_cats}
Let $C$ be the entrance path category of a finite CW complex $X$ considered as an ordinary category, i.e., with the poset structure of the Hom sets removed. Then $C$ is a unique factorization category. To see this, observe that each non-identity morphism goes to a cell of strictly lower dimension, so $C$ is finite directed. Furthermore, an entrance path $(x_0 > \dots > x_k)$ factors uniquely to $(x_k > x_{k-1}) \circ \dots \circ (x_0 > x_1)$.

Similarly, we can show that a discrete flow category of a finite CW complex, with the poset structure of the Hom sets removed, is a unique factorization category. The factorization of a gradient path is given by subsequences between critical cells. As each non-identity morphism goes from a critical cell to a critical cell of lower dimension, this gives a well-defined factorization into indecomposable morphisms. Now, to see that this factorization is unique, suppose we have $f, f' \colon a \to b$ and $g, g' \colon b \to c$ with $g \circ f = g' \circ f'$. Then the sequence representations of $g \circ f$ and $g' \circ f'$ are equal, and as $b$ is critical, both sequence representations contain $b$. Now, the subsequences starting at $a$ and ending at $b$ have to be equal, and these are precisely the sequence representations of $f$ and $f'$, so $f = f'$. Similarly, $g = g'$. This shows that factorizations are unique.

To give an example, in the discrete flow category in \autoref{ex:D3_flo}, the morphism $[f > t > z < b > x]$ has the unique factorization $[t > z < b > x] \circ [f > t]$.
\end{example}

We now prove \autoref{theoremC}.
\begin{theorem}[Theorem C]
    Let $C$ be a unique factorization category. Then the nerve of $C$ deformation retracts to a simplicial set whose $n$--simplices are all degenerate for $n > 1$.
\end{theorem}

\begin{proof}
    The main idea of the proof is to collapse $\N(C)$ by removing 1--simplices represented by higher compositions.

    We collapse $\N(C)$ inductively. Let $k$ be the largest integer such that there is a 1--simplex in $\N(C)$ represented by a morphism $f$ that decomposes into $k$ indecomposable morphisms, i.e., $f = f_k \circ \dots \circ f_1$. We claim that $f$ is a free face, and $f$ is in a free pair with the $k$--simplex $\{f_1, \dots, f_k\}$. To see why this is true, let $\{g_1, \dots, g_m\}$ be a coface of $f$. Then $g = g_m \circ \dots \circ g_1$ is a 1--simplex that factors through $f$, and thus $g = f$, as otherwise, $g$ would factor into more than $k$ indecomposable morphisms. It follows that each $g_i$ is a composite $g_i = f_{j_{i+1}-1} \circ \dots \circ f_{j_i}$, and hence, $\{g_1, \dots, g_m\}$ is a face of $\{f_1, \dots, f_k\}$.

    Now, we can remove any $k$--composable 1--simplex $f$ from the simplicial set through a series of collapses. Now, all remaining 1--simplices decomposes into $k-1$ or fewer indecomposable factors. We can repeat the process with $k-1$, and so on, until all remaining 1--simplices are indecomposable morphisms. Now, all remaining nondegenerate simplices are of dimension 0 or 1, as any higher simplex $\{f_1, \dots, f_k\}$ has the decomposable morphism $f_k \circ \dots \circ f_1$ as a face.
\end{proof}

By combining this theorem with \autoref{thm:simpl_homology_degen}, we get the following corollary.

\begin{corollary}\label{cor:ufc_homology}
    Let $C$ be a unique factorization category. Then
    \begin{equation*}
        H_n(\N(C)) \cong 0
    \end{equation*}
    for $n \ge 2$.
\end{corollary}

\section{Spectral sequences}

We define a spectral sequence as a collection of bigraded differential abelian groups $\{E^r_{*,*}, \partial^r : r = 0,1,\dots\}$, where the bidegree of $\partial^r$ is $(-r, r-1)$. For a double complex $C$ we denote its horizontal differentials with $\partial_h \colon C_{*,*} \to C_{*-1,*}$ and its vertical differentials with $\partial_v \colon C_{*,*} \to C_{*,*-1}$. Furthermore, we let $\Tot C$ denote the total complex of $C$, with the convention that its differentials are given by $\partial^{\Tot}_n = \partial_v + (-1)^{\text{vertical degree}} \partial_h$.

We will use the two spectral sequences in the following theorem, which is the homological version of Theorem 2.15 in \cite{McCleary} (see also \cite[pp.\ 141--143]{Weibel} for the homological formulation).

\begin{theorem}\label{thm:dbl_cplx_spec_seq}
    Let $C$ be a double complex such that $C_{p,q} = 0$ whenever $p < 0$ or $q < 0$. Then there are two spectral sequences $_I E$ and $_{II} E$ with
    \begin{equation*}
        _I E^2_{p,q} \cong H_I H_{II} (C)_{p,q}, \quad \text{and} \quad _{II} E^2_{p,q} \cong H_{II} H_{I} (C)_{p,q},
    \end{equation*}
    that both converge to $H_*(\Tot C)$.
\end{theorem}

\subsection{The spectral sequence of a bisimplicial set}

To a bisimplicial set $X$, we can associate a double complex $FX$ with $(FX)_{p,q} = \Z X_{p,q}$. The horizontal and vertical differentials are induced by the horizontal and vertical face maps as for the chain complex of a simplicial set, i.e.,
\begin{align*}
    \partial_h &= \sum_{i} (-1)^{i} (d_i, \id), \\
    \partial_v &= \sum_{i} (-1)^{i} (\id, d_i).
\end{align*}
A simple calculation shows that $\partial_h^2 = 0$, $\partial_v^2 = 0$ and $\partial_h \partial_v = \partial_v \partial_h$.

The following theorem of Dold and Puppe tells us that the homology of $\Tot (FX)$ equals the homology of $\diag X$ \cite[Theorem 2.9]{DoldPuppe} (see also \cite[Theorem 2.5]{GoerssJardine}).
Hence, we can compute the homology of $\diag X$ with either of the spectral sequences in \autoref{thm:dbl_cplx_spec_seq}.

\begin{theorem}[Dold--Puppe]\label{thm:dold_puppe}
    Let $X$ be a bisimplicial set and $FX$ its associated double complex. Then
    \begin{equation*}
        H_{*} \Tot (FX) \cong H_{*} (\diag X)
    \end{equation*}
\end{theorem}

As a bisimplicial set has an infinite amount of degenerate simplices, computing the spectral sequence associated to $\Tot (FX)$ can be impractical. However, we will show that, just as for the chain complex of a simplicial set, we can in fact ignore the degenerate simplices (both the vertical and horizontal ones). For the proof, we will need the following lemma.

\begin{lemma}\label{lemma:subcplx_union}
    Let $C_{*}$ be a chain complex, and let $D_{*}$ and $D'_{*}$ be subcomplexes of $C_{*}$. Suppose 
    \begin{align*}
        H_n(D) &\cong 0,\\
        H_n(D') &\cong 0, \text{ and }\\
        H_n(D \cap D') &\cong 0
    \end{align*}
    for all $n \in \Z$. Then,
    \begin{equation*}
        H_n(D + D') \cong 0
    \end{equation*}
    for all $n \in \Z$.
\end{lemma}

Here, $D_* \cap D'_*$ and $D_* + D'_*$ are the subcomplexes of $C_*$ defined by
\footnote{Note the difference between $D_* + D'_*$ and the direct sum $D_* \oplus D'_*$.}
\begin{align*}
    (D \cap D')_n &= D_n \cap D'_n, \text{ and} \\
    (D + D')_n &= D_n + D'_n = \{x + y : x \in D_n, y \in D'_n\}.
\end{align*}

\begin{proof}
    There is a short exact sequence of chain complexes
    \begin{equation*}
        0 \rightarrow D_{*} \cap D'_{*} \rightarrow D_{*} \oplus D'_{*} \rightarrow D_{*} + D'_{*} \rightarrow 0,
    \end{equation*}
    where the first map is $x \mapsto (x, -x)$, and the second map is $(x, y) \mapsto x + y$. The induced long exact sequence in homology then proves the lemma.
\end{proof}

Now, let $D^h_{p,q}$ be the set of horizontal degeneracies in $X_{p,q}$, i.e., the simplices on the form $(s_i, \id) x$ for some $x \in X_{p-1, q}$ and some $i$. Likewise, let $D^v_{p,q}$ be the set of vertical degeneracies in $X_{p,q}$. Then $\Z D^h_{p,q}$ and $\Z D^v_{p,q}$ are subgroups of $(FX)_{p,q}$. Now, as $D^h_{*,q}$ are just the degeneracies in the simplicial set $X_{*,q}$, we have that $\partial_h (D^h_{p,q}) \subseteq D^h_{p-1,q}$ (as stated in \autoref{sec:simplicial_homology}). Furthermore, for $(s_i, \id) x \in D^h_{p,q}$,
\begin{equation}\label{eq:hor_ver_commute}
\begin{aligned}
    \partial_v (s_i, \id) x &= \left(\sum_{i} (-1)^{i} (\id, d_i)\right) (s_i, \id) x \\
    & = \sum_{i} (-1)^{i} (s_i, d_i) x = (s_i, \id) \left(\sum_{i} (-1)^{i} (\id, d_i)\right) x,
\end{aligned}
\end{equation}
so $\partial_v (D^h_{p,q}) \subseteq D^h_{p,q-1}$. Hence, $F D^h$ given by $(F D^h)_{p,q} = \Z D^h_{p,q}$ is a well defined sub-double complex of $FX$, and we can form the quotient double complex $FX / F D^h$. Likewise, we have well-defined double complexes $F D^v$ and $FX/F D^v$. We then also have a sub-double complex $F(D^h \cup D^v)$, again given by $F(D^h \cup D^v)_{p,q} = \Z (D^h_{p,q} \cup D^v_{p,q})$, and we can form the quotient $FX / F(D^h \cup D^v)$.

\begin{proposition}\label{thm:ignore_degeneracies}
    Let $X$ be a bisimplicial set. Then
    \begin{equation*}
        H_{*} \Tot (FX) \cong H_{*} \Tot \left(FX / F(D^h \cup D^v)\right).
    \end{equation*}
\end{proposition}

\begin{proof}
    From \autoref{cor:simpl_homology_degen}, we get that $H_I (FD^h) \cong 0$ everywhere. \autoref{thm:dbl_cplx_spec_seq} then gives us that $H_{*}(\Tot FD^h) \cong 0$. 
    By the same argument for $D^v$, we get that $H_{*} (\Tot FD^v) \cong 0$.

    Now, the degenerate simplices in the diagonal $\diag X$ are precisely the elements of $D^h \cap D^v$. Applying \autoref{cor:simpl_homology_degen} again, we get that $H_{*}(\diag(D^h \cap D^v)) = 0$, and applying \autoref{thm:dold_puppe}, we get that $H_{*}(\Tot F(D^h \cap D^v)) \cong 0$.

    Note now that $\Tot F(D^h \cup D^v) = \Tot (F D^h + F D^v) = (\Tot F D^h) + (\Tot F D^v)$ and that $\Tot F(D^h \cap D^v) = (\Tot F D^h) \cap (\Tot F D^v)$.
    \sloppy It now follows from \autoref{lemma:subcplx_union} that $H_{*} (\Tot F(D^h \cup D^v)) \cong 0$. Hence, $H_{*} (\Tot FX / F(D^h \cup D^v)) \cong H_{*} \left((\Tot FX) / (\Tot F(D^h \cup D^v))\right) \cong H_{*} (\Tot FX)$.
\end{proof}

\subsection{The spectral sequence of the discrete flow category}

From the discrete flow category $\Flo_{\Sigma}[X]$, we get a bisimplicial set $S := \DN \Flo_{\Sigma}[X]$, as described in \autoref{sec:double_nerve}.
As the realization of the diagonal, $|\diag S|$, is (homotopy equivalent to) the classifying space of $\Flo_{\Sigma}[X]$, we can compute the homology of this classifying space through simplicial homology on $\diag S$. This is again equal to the homology of the total complex, $\Tot FS$, by \autoref{thm:dold_puppe}.
Finally, we can compute the homology of $\Tot FS$ with one of two the spectral sequences in \autoref{thm:dbl_cplx_spec_seq}. We will use the spectral sequence $_{I}E$, i.e., we will take vertical homology first. As we shall show, this spectral sequence collapses on page 2 (\autoref{theoremB} in the introduction).

Recall that an element of $S_{p,q}$ consists of $q$ horizontally composable sets of $p$ vertically composable 2--morphisms. For p-categories, this is a set $\{f_{0,j} \Rightarrow \dots \Rightarrow f_{p,j} : f_{i,j} \colon x_{j} \to x_{j+1}, 0 \le j < q \}$. The face and degeneracy maps, both in the vertical and in the horizontal direction, are as for the ordinary nerve.

First, we compute some examples.

\begin{example}

We consider the discrete flow category of a discrete Morse function on $S^2$, as described in \cite[Section 6.1]{Nanda}. Let $S$ be the double nerve of the discrete flow category. By \autoref{thm:ignore_degeneracies}, we need only consider the elements of $S$ that are not in the image of any degeneracy map (either vertical or horizontal) when constructing the spectral sequence.

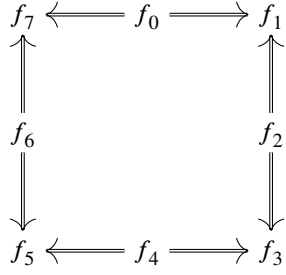
\begin{figure}[htbp]
    \centering
    \begin{tikzpicture}
        \diagram{d}{3em}{3em}{
            f_7 & f_0 & f_1 \\
            f_6 & \ & f_2 \\
            f_5 & f_4 & f_3 \\
        };
        
        \draw[->, double] (d-1-2) -- (d-1-3);
        \draw[->, double] (d-1-2) -- (d-1-1);
        \draw[->, double] (d-2-3) -- (d-1-3);
        \draw[->, double] (d-2-3) -- (d-3-3);
        \draw[->, double] (d-2-1) -- (d-1-1);
        \draw[->, double] (d-2-1) -- (d-3-1);
        \draw[->, double] (d-3-2) -- (d-3-3);
        \draw[->, double] (d-3-2) -- (d-3-1);
    \end{tikzpicture}
    \caption{The Hom poset $\Hom(t, w)$.}
    \label{fig:S2_hom_t_w}
\end{figure}

As the discrete flow category has only two objects, $t$ and $w$, there are no nondegenerate horizontal compositions, and thus $E^0_{0, q} = 0$ for $q \ge 2$. In the bottom row, we have $E^0_{0,0} = \Z^2$, corresponding to the two objects, and $E^0_{p, 0} = 0$ for $p \ge 1$. Furthermore, $\Hom(t, w)$ (illustrated in \autoref{fig:S2_hom_t_w}) has 8 nondegenerate 1-morphisms, 8 nondegenerate 2-morphisms (i.e., partial orders) and no nondegenerate compositions of morphisms, so $E^0_{0, 1} = E^0_{1, 1} = \Z^8$ and $E^0_{p, 1} = 0$ for $p \ge 2$.

\begin{figure}[htbp]
    \centering
    \begin{sseqpage}[title = Page 0]
        \class["\Z^2"](0,0)
        \class["\Z^8"](0,1)
        \class["\Z^8"](1,1)
        \d0 (0,1)
    \end{sseqpage}
    \quad
    \begin{sseqpage}[title = Page 1]
        \class["\Z"](0,0)
        \class["\Z^7"](0,1)
        \class["\Z^8"](1,1)
        \d1 (1,1)
    \end{sseqpage}
    \quad
    \begin{sseqpage}[title = Page 2]
        \class["\Z"](0,0)
        \class["\Z"](1,1)
    \end{sseqpage}
    \caption{The first three pages of the spectral sequence for the discrete flow category of a discrete Morse function on $S^2$.}
    \label{fig:S2_spectral_sequence}
\end{figure}

The 0th page, $E^0$, is illustrated in \autoref{fig:S2_spectral_sequence}, together with the two succeeding pages. To compute $E^1$, we need only find the differential $\partial^0_{0,1} \colon E^0_{0,1} \to E^0_{0,0}$. This image takes $f_i \colon t \to w$ to $(w - t)$, and hence its image is $\Z (1, -1) \cong \Z$. This gives us $E^1$, as illustrated.

To compute $E^2$, we need the differential $\partial^1_{1,1} \colon E^1_{1,1} \to E^1_{0,1}$. Observe that that $E^1_{0,1}$ is generated by $[f_i - f_0], i \in {1, \dots, 7}$. From this, it's clear that $\partial^1_{1,1}$ is surjective (for example, $[f_3 - f_0]$ is $\partial^1 \left[ (f_2 \Rightarrow f_3) - (f_2 \Rightarrow f_1) + (f_0 \Rightarrow f_1) \right] )$. Hence, we get $E^2$ as illustrated. The resulting homology of the total complex, and hence $S^2$, is then as follows.

\begin{equation*}
    H_n(S^2) =
    \begin{cases}
        \Z, &\quad n = 0,2, \\ 
        0, &\quad \text{ otherwise.} \\
    \end{cases}
\end{equation*}

\end{example}

\begin{example}

We compute a slightly more complicated example, the discrete flow category of the discrete Morse function on $D^3$ from \autoref{ex:D3_flo}. There are three critical cells: $f$, $t$ and $x$. We computed the Hom poset $\Hom(f, x)$ in \autoref{ex:D3_flo}. The Hom poset $\Hom(t, x)$ is as for the $S^2$ example, i.e., the one illustrated in \autoref{fig:S2_hom_t_w}. Finally, $\Hom(f, t)$ consists of just one morphism, $(f > t)$.

\begin{figure}[htbp]
    \centering
    \begin{sseqpage}[title = Page 0]
        \class["\Z^3"](0,0)
        \class["\Z^{30}"](0,1)
        \class["\Z^8"](0,2)
        \class["\Z^{60}"](1,1)
        \class["\Z^8"](1,2)
        \class["\Z^{32}"](2,1)
        \d0 (0, 2)
        \d0 (0, 1)
        \d0 (1, 2)
    \end{sseqpage}
    \ 
    \begin{sseqpage}[title = Page 1]
        \class["\Z"](0,0)
        \class["\Z^{20}"](0,1)
        \class["\Z^{52}"](1,1)
        \class["\Z^{32}"](2,1)
        \d1 (1, 1)
        \d1 (2, 1)
        \class(2,2) 
    \end{sseqpage}
    \  
    \begin{sseqpage}[title = Page 2]
        \class["\Z"](0,0)
        \class(2,2) 
    \end{sseqpage}
    \caption{The first three pages of the spectral sequence for the discrete flow category of a discrete Morse function on $D^3$.}
    \label{fig:D3_spectral_sequence}
\end{figure}

We count the nondegenerate simplices to get $E^0$. In $E_{p, 0}$, as in the previous example, we get the $E^0_{p, 0} = 0$ for $p \ge 1$, and $E^0_{0,0} = \Z^3$, one copy of $\Z$ for each of the critical cells. There are 30 non-identity morphisms, so $E^0_{0,1} = \Z^{30}$. There are 60 non-identity partial order relations, so $E^0_{1,1} = \Z^{60}$. The only Hom poset with composable partial orders is $\Hom(f,x)$, and there are 32 possible nondegenerate pairs of composable partial orders, so $E^0_{2,1} = \Z^{32}$. Finally, there are no nondegenerate triples of composable partial orders, so $E^0_{p, 1} = 0$ for $p \ge 3$.

The only non-identity (horizontally) composable morphisms are $(f > t)$ composed with one of the eight morphisms in $\Hom(t, x)$, so $E^0_{0,2} = \Z^8$. Similarly, the only nondegenerate horizontally composable partial orders are $\left( (f > t) \Rightarrow (f > t) \right)$ composed with one of the eight non-identity partial orders in $\Hom(t, x)$\footnote{Even though the partial order $\left( (f > t) \Rightarrow (f > t) \right)$ is degenerate, its horizontal composition with a nondegenerate partial order is not.}. Thus, $E^0_{1,2} = \Z^8$. Now, all elements in $S_{2,2}$ includes either a vertical identity (on $(f > t)$) or a horizontal identity (on $f$, $t$ or $x$), and is thus degenerate, so $E^0_{2,2} = 0$, and the same holds for $E^0_{p,2}$ for $p \ge 3$.

Finally, any horizontal composition of 3 or more must include an identity (as we only have 3 critical cells), so $E^0_{p,q} = 0$ for $q \ge 3$. We can now draw the 0th page of the spectral sequence. The first three pages are shown in \autoref{fig:D3_spectral_sequence}.

To get $E^1$ from $E^0$, me must inspect the three differentials
\begin{align*}
    &\partial^0_{0,1} \colon E^0_{0,1} \to E^0_{0,0}, \\
    &\partial^0_{0,2} \colon E^0_{0,2} \to E^0_{0,1}, \text{ and }\\
    &\partial^0_{1,2} \colon E^0_{1,2} \to E^0_{1,1}.
\end{align*}
The image of $\partial^0_{0,1}$ is the span of $\{f - t, f - x, t - x\}$, so its cokernel is $\Z$. Now, $\partial^0_{0,2}$ takes a composable pair $\{f, g\}$ to $(f + g - f \circ g)$. Now, as $f \circ g$ is different for each of the composable pairs $\{f, g\}$ in $E^0_{0,2}$, the kernel of $\partial^0_{0,2}$ is 0. Furthermore, we see that computing the cokernel of $\partial^0_{0,2}$ amounts to identifying $f \circ g$ with $f + g$ for each of these 8 compositions, so the cokernel is torsion-free and $E^1_{0,1} = \Z^{30-8-2} = \Z^{20}$.
By a similar argument, the kernel of $\partial^0_{1,2}$ is 0 and its cokernel is $\Z^{52}$.

Computing $E^2$ from $E^1$ can be done through tedious computation. Working out this computation gives that $\partial^1_{1,1}$ is surjective and that $\Ima \partial^1_{2,1} = \ker \partial^1_{1,1}$, so $E^2_{0,1}, E^2_{1,1}$ and $E^2_{2,1}$. The spectral sequence now collapses, and we get that the computed homology $\Z$ in degree 0 and 0 everywhere else, as expected from $D^3$.

\end{example}

The key takeaway from this example is that all rows above $q=1$ became 0 everywhere on the $E^1$--page. As we shall now show, this always happens, which causes the spectral sequence to collapse on the second page. For the rest of the section, when we write a \emph{discrete flow category}, we shall mean a discrete flow category induced from a discrete Morse function on a regular CW complex, as defined in \autoref{sec:discrete_flow_category}.

For the following result, recall that for a p-category $C$, $\NN C$ is the simplicial category described in \autoref{sec:double_nerve}.

\begin{theorem}\label{thm:nn_dfc_is_ufc}
    Let $X$ be a regular CW complex and let $\Sigma$ be a Morse system induced from a discrete Morse function.
    Then $(\NN \Flo_{\Sigma}[X]) ([n])$ is a unique factorization category for all $n \ge 0$.
\end{theorem}

\begin{proof}
    As observed in \autoref{ex:unique_factorization_cats}, $(\NN \Flo_{\Sigma}[X])([0])$, which is just $\Flo_{\Sigma}[X]$ with the poset structure removed, is a unique factorization category.

    Now, let $n \ge 1$. The morphisms in $(\NN \Flo_{\Sigma}[X]) ([n])$ are represented by $f_0 \Rightarrow \dots \Rightarrow f_n$ with $f_i$ morphisms in $\Flo_{\Sigma}[X]$. It is clear that $(\NN \Flo_{\Sigma}[X]) ([n])$ is a finite directed category. Furthermore, given a representative $f_0 \Rightarrow \dots \Rightarrow f_n$, $f_0 \colon x \to y$ has a unique factorization through critical cells $x = c_0, c_1, \dots, c_k = y$. By \autoref{thm:flo_po_desc}, then all $f_i$ factors through these critical cells. Let $f_i = f_i^{1} \circ \dots \circ f_i^{k}$ be the the decomposition of $f_i$ through $c_0, \dots, c_k$. We then have a decomposition $\{f_0, \dots, f_n\} = \{f_0^1, \dots, f_n^1\} \circ \dots \circ \{f_0^k, \dots, f_n^k\}$. The sets on the right hand side are indecomposable, as $f_0^i$ is indecomposable for all $i$. Furthermore, the decomposition is unique, as $f_0^1 \circ \dots f_0^k$ is the unique decomposition of $f_0$, and all other $f_i^j$ are decided uniquely by this decomposition.
\end{proof}

\begin{corollary}
    Let $\{E^r_{*,*}\}$ be the spectral sequence associated to a discrete flow category. 
    Then $E^1_{p,q} = 0$ for $q \ge 2$ 
\end{corollary}

\begin{proof}
    Let $C$ be the discrete flow category.
    By definition, $E^1_{p,q}$ is the $q$th homology of the simplicial set $E^0_{p,*} = (\DN C)_{p, *} = \N\left((\NN C)([p])\right)$. By combining \autoref{thm:nn_dfc_is_ufc} and \autoref{cor:ufc_homology}, we get that this is 0 when $q \ge 2$.
\end{proof}

\begin{lemma}
    Let $\{E^r_{*,*}\}$ be the spectral sequence associated to a discrete flow category. 
    Then $E^2_{p,0} = 0$ for $p \ge 1$.
\end{lemma}

\begin{proof}
    Let $C$ be the discrete flow category.
    We have that $E^1_{p,0}$ is the free abelian group on the connected components of $\N\left((\NN C)([p])\right)$. Now, if $a$ and $b$ are connected in $\N\left((\NN C)([p])\right)$, then they are connected in $\N\left((\NN C)([i])\right)$ for all $i$.
    Hence, $E^1_{p,0} = \{(s_0)^p [x] : [x] \in E^1_{0,0}\}$. Now, as $\partial_h (s_0)^p [x]$ is $[x]$ when $p$ is odd and 0 when $p$ is even, taking the horizontal homology gives us 0 everywhere except in $E^1_{0,0}$.
\end{proof}

With the previous two results, we are finally ready to prove \autoref{theoremB}.

\begin{theorem}[Theorem B]
    The spectral sequence associated to a discrete flow category collapses on the second page.
\end{theorem}

\begin{proof}
    As $E^1_{p,q} = 0$ for $q \ge 2$, the same applies to $E^2$. Hence, the only potential nonzero entries on page 2 are $E^2_{0,0}$ and $E^2_{p,1}, p\ge0$. Hence, there can be no nonzero differentials from page 2 and onward.
\end{proof}

Note that the above results applies to the ordinary spectral sequence associated to a bisimplicial set, \emph{not} the one where degenerate simplices have been removed (as in \autoref{thm:ignore_degeneracies}). However, for computational purposes it is useful to remove degenerate simplices from the spectral sequence, as there are always infinitely many degenerate simplices. Therefore, we now show that the spectral sequence also collapses on page 2 when removing degeneracies.

\begin{theorem}\label{thm:reduced_spec_seq_collapse}
    Let $X$ be the double nerve of a discrete flow category. Then the spectral sequence associated to the double complex
    \begin{equation*}
        FX / F(D^h \cup D^v)
    \end{equation*}
    collapses at the second page.
\end{theorem}

\begin{proof}
    Let $\{E^r_{*,*}\}$ be the associated spectral sequence. Firstly, there are no nondegenerate simplices in $X_{p,0}$ for $p \ge 1$, so $E^0_{p,0} \cong 0$ for $p \ge 1$ (and the same holds for the succeeding pages).

    We now show that $E^1_{p,q} \cong 0$ for $q \ge 2$. Recall that $E^1_{p,q} = H_{II}(FX/F(D^h \cup D^v))_{p,q}$ First, we show that $H_{II}(FX/FD^h)_{p,q} \cong 0$ for $q \ge 2$. Suppose that $[x] \in (FX/FD^h)_{p,q}$, with $q \ge 2$, such that $\partial_v [x] = 0$. Then $x \in FX_{p,q}$ is such that $\partial_v x \in FD^h$, so $\partial_v x = s^h_0 y_0 + \dots + s^h_{p-1} y_{p-1}$ for some $y_0, \dots, y_{p-1} \in FX_{p-1,q}$, where $s_i^h = (s_i, \id)$ is the $i$th horizontal degeneracy map. Now, let $x' = x - s^h_{p-1} d^h_p x$. Then $[x'] = [x]$ in $(FX/FD^h)_{p,q}$, and
    \begin{align*}
        \partial_v x' &= \partial_v x - \partial_v (s^h_{p-1} d^h_p) x
        = \partial_v x - (s^h_{p-1} d^h_p) \partial_v x \\
        &= \left(s^h_0 y_0 + \dots + s^h_{p-1} y_{p-1}\right) - \left( s^h_{p-1} d^h_p s^h_0 y_0 + \dots + s^h_{p-1} d^h_p s^h_{p-1} y_{p-1} \right).
    \end{align*}
    We now utilize the fact that $d_p s_{p-1} = \id$ and that when $i < p-1$, $s_{p-1} d_p s_i = s_{p-1} s_i d_{p-1} = s_i s_{p-2} d_{p-1}$. Thus, the sum above rewrites to
    \begin{align*}
        s^h_0 & (y_0 - s^h_{p-2} d^h_{p-1} y_0) + \dots + s^h_{p-2} (y_{p-2} - s^h_{p-2} d^h_{p-1} y_{p-2}) + s^h_{p-1} (y_{p-1} - y_{p-1}) \\
        & = s^h_0 y_0' + \dots s^h_{p-2} y_{p-2}'.
    \end{align*}
    Now, continuing this process by letting $x'' = x' - s^h_{p-2} d^h_{p-1} x'$, and so on, we end up with a $\hat x \in FX_{p,q}$ such that $[\hat x] = [x] \in (FX/FD^h)_{p,q}$ and $\partial_v \hat x = 0$. But as $H_{II}(FX)_{p,q}$ is 0 when $q \ge 2$, then $\hat x$ is in the image of $\partial_v$, and thus $[x]$ is in the image of $\partial_v$. Hence, $H_{II}(FX/FD^h)_{p,q} \cong 0$ when $q \ge 2$.

    It remains to show that $H_{II}(FX/F(D^h \cup D^v))_{p,q} \cong 0$ when $q \ge 2$. For this, we use that $G/(AB) = (G/A)/(B/(A \cap B))$ for groups $A,B < G$, so we have a short exact sequence of chain complexes
    \begin{equation}\label{eq:hor_ver_degen_ses}
        0 \rightarrow FD^v/(FD^h\cap FD^v) \rightarrow FX/FD^h \rightarrow FX/F(D^h \cup D^v) \rightarrow 0.
    \end{equation}
    We know from \autoref{cor:simpl_homology_degen} that $H_{II}(FD^v) \cong 0$. Furthermore, one can show that $D^h_{p,*}$ together with the vertical face and degeneracy maps is a simplicial set\footnote{This follows from the fact that if $x$ is a horizontal degeneracy, then so is $d^v_i x$ and $s^v_i x$, as $d^v_i s^h_j = s^h_j d^v_i$ and $s^v_i s^h_j = s^h_j s^v_i$.}.
    The degeneracies in this simplicial set are $D^h \cap D^v$, so, again by \autoref{cor:simpl_homology_degen}, $H_{II}(F(D^h \cap D^v)) = H_{II}(FD^h\cap FD^v)) \cong 0$. In then follows (from the long exact sequence in homology of short exact sequences of chain complexes) that $H_{II}(F^v/(FD^h\cap FD^v)) = 0$.

    Finally, we can use the long exact sequence in homology from the short exact sequence in \eqref{eq:hor_ver_degen_ses} to conclude that $H_{II}(FX/F(D^h \cup D^v))_{p,q} \cong H_{II}(FX/FD^h)_{p,q}$ (for all $p$ and $q$). In particular, $H_{II}(FX/F(D^h \cup D^v))_{p,q} \cong 0$ for $q \ge 2$.

    Hence, $E^1_{p,q}$ is 0 when $q \ge 2$, and this, together with the fact that $E^1_{p,0} \cong 0$ when $p \ge 1$, shows that the spectral sequence collapses on the second page.
\end{proof}

Observe that due to this theorem, when computing the spectral sequence, we never need to know $E^0_{p,q}$ for $q \ge 3$. In fact, the only things we need to compute at page 1 is the row $E^1_{*,1}$ and the single square $E^1_{0,0}$. The squares in $E^1_{*,1}$ also have a relatively simple description; the differential $\partial_v E^0_{p,2} \to E^0_{p,1}$ sends a composable pair $\{f, g\}$ to $f + g - g \circ f$. Hence, $E^1_{p,1}$ is the same as $E^0_{p,1}$ where all morphisms have been identified with the sum of the morphisms in their (unique) decomposition of indecomposable morphisms.

We now show another example of a spectral sequence computation, using \autoref{thm:reduced_spec_seq_collapse} and the observations in the last example.

\begin{example}
We compute the discrete flow category of a discrete Morse function on the 2--torus, illustrated in \autoref{fig:torus_dmf}. We compute the associated spectral sequence (with degeneracies removed), with rational coefficients (to make the computations easier).

\begin{figure}[htbp]
    \centering
    \begin{tikzpicture}[scale=1]
        \filldraw[lightgray] (0,0) -- (4,0) -- (4, 4) -- (0, 4) -- cycle;
        \draw (1, 3) node{$A$};
        \draw (3, 3) node{$B$};
        \draw (1, 1) node{$C$};
        \draw (3, 1) node{$D$};
        
        \draw[thick, ->-] (0, 4) -- (2, 4) node[midway, anchor=south]{$\alpha$};
        \draw[thick, ->-] (2, 4) -- (4, 4) node[midway, anchor=south]{$\beta$};
        \draw[thick, ->-] (0, 0) -- (2, 0) node[midway, anchor=north]{$\alpha$};
        \draw[thick, ->-] (2, 0) -- (4, 0) node[midway, anchor=north]{$\beta$};
        
        \draw[thick] (0, 2) -- (2, 2) node[midway, anchor=south]{$\gamma$};
        \draw[thick] (2, 2) -- (4, 2) node[midway, anchor=south]{$\delta$};
        
        \draw[thick, ->-] (0, 4) -- (0, 2) node[midway, anchor=east]{$\varepsilon$};
        \draw[thick, ->-] (0, 2) -- (0, 0) node[midway, anchor=east]{$\zeta$};
        \draw[thick, ->-] (4, 4) -- (4, 2) node[midway, anchor=west]{$\varepsilon$};
        \draw[thick, ->-] (4, 2) -- (4, 0) node[midway, anchor=west]{$\zeta$};
        
        \draw[thick] (2, 4) -- (2, 2) node[midway, anchor=east]{$\eta$};
        \draw[thick] (2, 2) -- (2, 0) node[midway, anchor=east]{$\theta$};
        
        \filldraw (0, 4) circle (1.5 pt) node[anchor=south east]{$a$};
        \filldraw (2, 4) circle (1.5 pt) node[anchor=south]{$b$};
        \filldraw (4, 4) circle (1.5 pt) node[anchor=south west]{$a$};
        \filldraw (0, 2) circle (1.5 pt) node[anchor=east]{$c$};
        \filldraw (2, 2) circle (1.5 pt) node[anchor=south east]{$d$};
        \filldraw (4, 2) circle (1.5 pt) node[anchor=west]{$c$};
        \filldraw (0, 0) circle (1.5 pt) node[anchor=north east]{$a$};
        \filldraw (2, 0) circle (1.5 pt) node[anchor=north]{$b$};
        \filldraw (4, 0) circle (1.5 pt) node[anchor=north west]{$a$};

        \draw[red, ultra thick, dashed, ->] (2, 4) -- (3, 4);
        \draw[red, ultra thick, dashed, ->] (2, 3) -- (3, 3);
        \draw[red, ultra thick, dashed, ->] (2, 2) -- (3, 2);
        \draw[red, ultra thick, dashed, ->] (2, 1) -- (3, 1);
        \draw[red, ultra thick, dashed, ->] (2, 0) -- (3, 0);
        \draw[red, ultra thick, dashed, ->] (1, 2) -- (1, 1);
        \draw[red, ultra thick, dashed, ->] (0, 2) -- (0, 1);
        \draw[red, ultra thick, dashed, ->] (4, 2) -- (4, 1);
    \end{tikzpicture}
    \caption{A discrete Morse function (represented by its gradient vector field) on a regular CW decomposition of the 2--torus. The regular pairs are given by red, dashed arrows.}
    \label{fig:torus_dmf}
\end{figure}
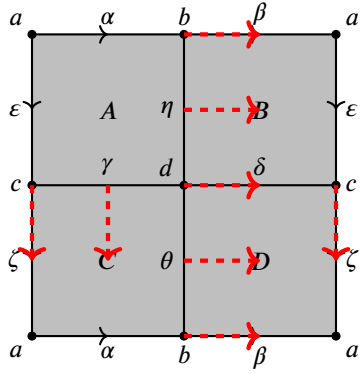

The critical cells are $A$, $\alpha$, $\epsilon$ and $a$. We compute the Hom posets with \autoref{algorithm_1}, and get:
\begin{align*}
    \Hom(A, \alpha) &= \left\{
        f_1 = [A > \alpha], 
        \ f_2 = [A > \gamma < C > \alpha]
    \right\}, \\
    \Hom(A, \varepsilon) &= \left\{
        f_1' = [A > \varepsilon], 
        \ f_2' = [A > \eta < B > \varepsilon]
    \right\}, \\
    \Hom(\alpha, a) &= \left\{
        g_1 = [\alpha > a], 
        \ g_2 = [\alpha > b < \beta > a]
    \right\}, \\
    \Hom(\varepsilon, a) &= \left\{
        g_1' = [\varepsilon > a], 
        \ g_2' = [\varepsilon > c < \zeta > a]
    \right\}.
\end{align*}
The only Hom poset with nontrivial partial orders is $\Hom(A,a)$, which consists of 36 morphisms in four connected components, as illustrated in \autoref{fig:Hom_A_a}.

\begin{figure}[htbp]
    \centering
    \begin{align*}
        & g_1 \circ f_1 \Rightarrow \ \bullet \ \Leftarrow g_1' \circ f_1' \\ \\
        & g_2 \circ f_1 \Rightarrow \ \bullet \ \Leftarrow \quad \dots \quad \Rightarrow \ \bullet \ \Leftarrow g_1' \circ f_2' \\ \\
        & g_1 \circ f_2 \Rightarrow \ \bullet \ \Leftarrow \quad \dots \quad \Rightarrow \ \bullet \ \Leftarrow g_2' \circ f_1' \\ \\
        & g_2 \circ f_2 \Rightarrow \ \bullet \ \Leftarrow \quad \dots \quad \Rightarrow \ \bullet \ \Leftarrow g_2' \circ f_2'
    \end{align*}
    \caption{The four components in the Hom poset $\Hom(A, a)$. The unnamed morphisms (represented by dots) are all indecomposable.}
    \label{fig:Hom_A_a}
\end{figure}

\begin{figure}[htbp]
    \centering
    \begin{sseqpage}[title = Page 0]
        \class["\Q^4"](0,0)
        \class["\Q^{44}"](0,1)
        \class["\Q^8"](0,2)
        \class["\Q^{32}"](1,1)
        \d0 (0, 2)
        \d0 (0, 1)
    \end{sseqpage}
    \quad
    \begin{sseqpage}[title = Page 1]
        \class["\Q"](0,0)
        \class["\Q^{33}"](0,1)
        \class["\Q^{32}"](1,1)
        \d1 (1, 1)
        \class(1,2) 
    \end{sseqpage}
    \quad
    \begin{sseqpage}[title = Page 2]
        \class["\Q"](0,0)
        \class["\Q^2"](0,1)
        \class["\Q"](1,1)
        \class(1,2) 
    \end{sseqpage}
    \caption{The first three pages of the spectral sequence for a discrete flow category of the 2--torus.}
    \label{fig:T2_spectral_sequence}
\end{figure}

The discrete flow category has 4 objects, 44 morphisms, 8 nondegenerate horizontal compositions of morphisms, 32 non-identity partial orders, and no other nondegenerate compositions. Hence, the $E^0$ page is as illustrated in \autoref{fig:T2_spectral_sequence}.

On page 1, $E^1_{0,0}$ is $\Q$, as the nerve of the category has one connected component. Hence, $E^1_{0,1} = \Q^{44-8-3} = \Q^{33}$. Furthermore, $E^1_{0,1}$ can be described as the cycles in $E^0_{0,1}$, where all compositions $f \circ g$ has been identified with $f + g$.

Now, to compute $E^2$, we need only determine the map $\partial^2_{1,1} \colon E^1_{1,1} \to E^1_{0,1}$. For this, observe that for each of the four components in \autoref{fig:Hom_A_a}, we can assign an alternating sum of partial orders $x_i$, so that
\begin{align*}
    \partial^2 x_1 = g_1' \circ f_1' - g_1 \circ f_1 = g_1' + f_1' - g_1 - f_1, \\
    \partial^2 x_2 = g_1' \circ f_2' - g_2 \circ f_1 = g_1' + f_2' - g_2 - f_1, \\
    \partial^2 x_3 = g_2' \circ f_1' - g_1 \circ f_2 = g_2' + f_1' - g_1 - f_2, \\
    \partial^2 x_4 = g_2' \circ f_2' - g_2 \circ f_2 = g_2' + f_2' - g_2 - f_2. \\
\end{align*}
Now, $\partial^2 (x_1 - x_2 - x_3 + x_4) = 0$. One can verify (for example through writing $\partial^2$ as a matrix and row reducing) that $(x_1 - x_2 - x_3 + x_4)$ generates the kernel of $\partial^2$, so that $E^2_{1,1} = \Q$ and $E^2_{0,1} = \Q^2$. The spectral sequence now collapses, as predicted by \autoref{thm:reduced_spec_seq_collapse}, and the computed homology of the 2--torus is:
\begin{equation*}
    H_n(T^2; \Q) =
    \begin{cases}
        \Q, &\quad n = 0,2, \\ 
        \Q^2, &\quad n = 1, \\ 
        0, &\quad \text{ otherwise.} \\
    \end{cases}
\end{equation*}

Observe that in this case we got a nonzero group in $E^2_{1,1}$ even though the homology of (the nerve of) the Hom poset $\Hom(A, a)$ is 0 in degree 1. This happens because on page 1, the compositions $f \circ g$ becomes identified with the sums $f + g$, so that the differential $\partial^2 \colon E^1_{1,1} \to E^1_{0,1}$ has a nonzero element in the kernel, even though the differential $\partial_h \colon E^0_{1,1} \to E^0_{0,1}$ does not.

\end{example}

\newpage
\section*{Declarations}
\subsection*{Ethical approval}
Not applicable.
\subsection*{Funding}
Not applicable.
\subsection*{Availability of data and materials}
Not applicable.

\bibliographystyle{plain}
\bibliography{ref}

\end{document}